\documentclass[12pt,oneside,english]{amsart}
\usepackage[abs]{overpic}
\usepackage{xcolor,varwidth}
\usepackage[T1]{fontenc}
\usepackage[latin9]{luainputenc}
\usepackage{geometry}
\usepackage{babel}
\usepackage{units}
\usepackage{ifsym}
\usepackage{amstext}
\usepackage{amsthm}
\usepackage{amssymb}
\usepackage{graphicx}
\usepackage[unicode=true,pdfusetitle,
 bookmarks=true,bookmarksnumbered=false,bookmarksopen=false,
 breaklinks=false,pdfborder={0 0 1},backref=false,colorlinks=false]
 {hyperref}

\definecolor{myred}{RGB}{255,0,0}
\definecolor{mygreen}{RGB}{0,128,0}
\definecolor{myblue}{RGB}{0,0,255}
\definecolor{mypurple}{RGB}{128,0,128}
\definecolor{mygray}{RGB}{153,153,153}

\makeatletter
\numberwithin{equation}{section}
\numberwithin{figure}{section}
\theoremstyle{plain}
\newtheorem{thm}{\protect\theoremname}[section]
\theoremstyle{definition}
\newtheorem{defn}[thm]{\protect\definitionname}
\theoremstyle{remark}
\newtheorem{rem}[thm]{\protect\remarkname}
\theoremstyle{plain}
\newtheorem{prop}[thm]{\protect\propositionname}
\theoremstyle{plain}
\newtheorem{conjecture}[thm]{\protect\conjecturename}
\theoremstyle{plain}
\newtheorem{cor}[thm]{\protect\corollaryname}
\theoremstyle{plain}
\newtheorem{lem}[thm]{\protect\lemmaname}
\theoremstyle{definition}
\newtheorem{example}[thm]{\protect\examplename}

\makeatother

\providecommand{\conjecturename}{Conjecture}
\providecommand{\corollaryname}{Corollary}
\providecommand{\definitionname}{Definition}
\providecommand{\examplename}{Example}
\providecommand{\lemmaname}{Lemma}
\providecommand{\propositionname}{Proposition}
\providecommand{\remarkname}{Remark}
\providecommand{\theoremname}{Theorem}

\begin{document}

\title[Legendrian large cables and non-uniformly
thick knots]{Legendrian large cables and new phenomenon for non-uniformly
thick knots}

\author{Andrew McCullough}

\begin{abstract}
We define the notion of a knot type having Legendrian large cables
and show that having this property implies that the knot type is not
uniformly thick. Moreover, there are solid tori in this knot type
that do not thicken to a solid torus with integer sloped boundary
torus, and that exhibit new phenomena; specifically, they have virtually overtwisted contact structures.
We then show that there exists an infinite family of ribbon knots
that have Legendrian large cables. These knots fail to be uniformly
thick in several ways not previously seen. We also give a general
construction of ribbon knots, and show when they give similar such
examples. 
\end{abstract}

\maketitle

\section{Introduction}

The \textit{contact width} $w\left(K\right)$ of a knot $K\subset\left(S^{3},\xi_{std}\right)$
was defined in \cite{key-1}, more or less as the largest slope of
a characteristic foliation on the boundary of a solid torus representing
the knot type $K$. They also defined $K$ to have the \textit{uniform thickness property} if any solid torus representing the knot type $\mathcal{K}$ can be thickened to a standard neighborhood of a Legendrian
representative of $K$ and $w\left(K\right)$ is equal to the maximal
Thurston-Bennequin invariant $\overline{tb}\left(K\right)$ of Legendrian
representatives of $K$. The usefulness of this property became evident
when Etnyre-Honda showed in the same work that, if $L\subset S^{3}$
is Legendrian simple and uniformly thick, then cables of $L$ are
Legendrian simple as well.  Recall that a knot type is Legendrian simple if Legendrian knots in this knot type are completely determined (up to Legendrian isotopy) by their Thurston-Bennequin invariant and rotation number.
They also showed that, if the cables are sufficiently negative, then
they too satisfy the uniform thickness property. This allows that
certain iterated cables of Legendrian simple knots are Legendrian
simple, for example. 

Uniform thickness has become a key hypothesis in work since then. For example, generalizing the above work on cables, in  \cite{key-9},
Etnyre-Vertesi showed that given a companion knot $L\subset S^{3}$
which is both Legendrian simple and uniformly thick, and a pattern
$P\subset S^{1}\times D^{2}$ satisfying certain symmetry hypothesis,
the knots in the satellite knot type $P_{K}$ may be understood. 

Broadly, if one wants to classify Legendrian knots in a satellite
knot type with companion knot $K\subset S^{3}$, and a pattern $P\subset S^{1}\times D^{2}$,
then as a first step one needs to understand,
\begin{enumerate}
\item contact structures on the complement of a neighborhood $N$ of $K$
\item contact structures on a neighborhood $N$ of $K$
\item a classification of Legendrian knots in the knot type of the pattern
$P$ in the possible contact structures on $N$.
\end{enumerate}
If $K$ is uniformly thick, then $N$ can always be taken to be a
standard neighborhood of $K$ with dividing curves on the boundary
of slope $\overline{tb}\left(K\right)$ (i.e. maximal Thurston-Bennequin
invariant of $K$), which reduces the problem to items $(1)$ and
$(3)$ above. Moreover, if $K$ is Legendrian simple and uniformly
thick, then $(1)$ is more or less known as well. If $K$ is not uniformly
thick, then understanding satellites is much more complicated. 

Similarly, uniform thickness can be useful in understanding contact
surgery constructions. A typical way to obtain a new contact 3-manifold is removing a solid torus in the knot type K, and gluing in some new contact solid torus. To understand the new manifold, one needs to understand items
$(1)$ and $(2)$ above, and the gluing map defining the surgery.
If $K$ is uniformly thick, then $N$ can always be taken to be a
standard neighborhood of $K$ with dividing curves on the boundary
of slope $\overline{tb}\left(K\right)$, which simplifies $(1)$ and
$(2)$ considerably.

On the other hand, there are knot types that are not uniformly thick.
For such knot types, it is important to understand in what ways they
can fail to be uniformly thick.

\subsection{New phenomenon for non-uniformly thick knots}

Given a knot type $\mathcal{K}\subset S^{3}$, the \textit{contact
width} of $\mathcal{K}$ is 
\[
w\left(\mathcal{K}\right)=\sup\left\{ slope\left(\varGamma_{\partial N}\right)\mid N\;solid\;torus\;representing\;\mathcal{K}\;with\;convex\;boundary\right\} .
\]
We say a solid torus represents $\mathcal{K}$ if its core is in the
knot type of $\mathcal{K}$. The contact width satisfies the inequality
$\overline{tb}\left(\mathcal{K}\right)\leq w\left(\mathcal{K}\right)\leq\overline{tb}\left(\mathcal{K}\right)+1$
\cite{key-1}.

A word about slope conventions. If $\mu,\lambda$ are the meridional,
respectively longitudinal, curves on a torus $T$ then $\left[\lambda\right]$,
$\left[\mu\right]$ form a basis for $H_{1}\left(T\right)$. A
$\left(p,q\right)$ curve, or a curve of slope $\nicefrac{q}{p}$,
will refer to any simple closed curve in $T$ that is in the homology
class of $p\left[\lambda\right]+q\left[\mu\right]$, where $p,q\in\mathbb{Z}$
are relatively prime. This is the opposite convention to the one used
in several of the main references in this paper, which were some of
the first works in convex surface theory. However, it is the convention
that is standard in low-dimensional topology. We caution however that,
when the phrase ``integer slope'' is used, it would correspond to
the phrase ``one over integer slope'' in \cite{key-5,key-1,key-2}
among others. 

We are now in position to define uniform thickness. We say that a
knot type $\mathcal{K}$ has the \textit{uniform thickness property}
or is \textit{uniformly thick} if
\begin{enumerate}
\item $\overline{tb}\left(\mathcal{K}\right)=w\left(\mathcal{K}\right)$,
and 
\item every solid torus representing $\mathcal{K}$ can be thickened to
a standard neighborhood of a maximal $tb$ representative of $\mathcal{K}$.
\end{enumerate}
By a standard neighborhood of a Legendrian knot $L$, we mean a solid
torus neighborhood $N$ of $L$ with convex boundary, dividing set
$\varGamma_{\partial N}$ consisting of two curves, with slope $tb\left(L\right)$. 

In past work, a knot type $\mathcal{K}$ can fail to have the uniform
thickness property in two ways. It can have neighborhoods whose slopes
are larger than $\overline{tb}$, as is the case with the unknot $U$,
which has $\overline{tb}\left(U\right)=-1$ and $w\left(U\right)=0$.
It can also happen that there are neighborhoods with slope strictly
less than $\overline{tb}$, but that do not thicken. The first and
only such examples are in \cite{key-1} and \cite{key-3} where it
is shown that all positive torus knots $T_{p,q}$ have tori $N$ with
slopes satisfying $slope\left(\varGamma_{\partial N}\right)<\overline{tb}\left(T_{p,q}\right)$
but that do not thicken. Moreover, the contact structure on all of
these $N$ is universally tight. 

In what follows we will denote the set of Legendrian knots, up to isotopy, in the
same topological knot type as $K$ by $\mathcal{L}\left(K\right)$.
We also use the convention that for a pair of relatively prime integers $p$ and $q$, the $\left(p,q\right)$ cable of $K$,
that is, the knot type of a curve of slope $\nicefrac{q}{p}$ on the
boundary of a torus neighborhood of $K$, is denoted by $K_{p,q}$.
Notice that if $p=\pm1$, then $K_{p,q}$ is a trivial cable in the
sense that it is isotopic to the underlying knot $K$. 
The following theorem of Etnyre-Honda motivates us to define some
new terminology.
\begin{thm}
(Etnyre-Honda, \cite{key-1}) If $K\subset S^{3}$ satisfies the uniform
thickness property, then for $\left|p\right|>1$ and any $L\in\mathcal{L}\left(K_{p,q}\right)$
we have that $tb\left(L\right)\leq pq$.\label{cables<pq}
\end{thm}

We generalize this result in Lemma~\ref{lem:twisting<=00003Dpq}
below. Notice that if we have a uniformly thick knot $K$ and we fix
a Legendrian representative $L\in\mathcal{L}\left(K\right)$ with
$tb\left(L\right)=k$, then there is an isotopy of $K$ which arranges
that $L$ is a trivial cable $L=K_{1,k-1}$. But then we have that
$tb\left(K_{1,k-1}\right)=tb\left(L\right)=k\nleqq k-1$, so the inequality
in Theorem~\ref{cables<pq} is not satisfied. 

\begin{defn}
Given $\left|p\right|>1$, we will say that a Legendrian cable $L\in\mathcal{L}\left(K_{p,q}\right)$
is \textit{large} if $tb\left(L\right)>pq$, and call $K_{p,q}$ \textit{Legendrian
large} if there exists large $L\in\mathcal{L}\left(K_{p,q}\right)$.
We will then say that $K$ has \textit{Legendrian large cables}, or
has the \textit{Legendrian large cable (LLC) property}, if any of
its non-trivial cables are Legendrian large. 
\end{defn}

Notice the example above indicates that if we allowed trivial cables, the LLC property would be vacuous. Our main theorem relates the LLC
property to uniform thickness.

\begin{thm}
If $K$ has Legendrian large cables, then there exist solid tori $V=S^{1}\times D^{2}$
representing $K$ such that $\xi\mid_{V}$ is virtually overtwisted.
Moreover, $V$ cannot be thickened to a standard neighborhood of a
Legendrian knot, and $K$ is not uniformly thick.\label{MainMainThm}
\end{thm}

Recall that the term \textit{universally tight} refers to a contact structure
that is tight, and that, when lifted to the universal cover, remains
tight. If the lift becomes overtwisted, then we will refer the the
contact structure as \textit{virtually overtwisted}. 

\begin{thm}
Given $K$, if there exists a slope $\nicefrac{q}{p}>\overline{tb}\left(K\right)$,
$\left|p\right|>1$, such that $K_{p,q}$ is Legendrian large, then
$w\left(K\right)>\overline{tb}\left(K\right)$.\label{MainMainThm2}
\end{thm}
\textbf{Question:} Are there knots $K$ and slopes $\nicefrac{q}{p}<\overline{tb}\left(K\right)$
such that $K_{p,q}$ is Legendrian large? 

\textbf{Question:} If $\xi$ is a virtually overtwisted contact structure
on $S=S^{1}\times D^{2}$, for which $p$ and $q$ is there a Legendrian
$\left(p,q\right)$ knot $L$ in $S$ with $tw\left(L\right)>pq$? 

The knots $K^{m}$ in Figure~\ref{Yasui knots} have $\overline{tb}\left(K^{m}\right)=-1$.
Building on the work of Yasui \cite{key-6}, we observe that $K_{\left(-n,1\right)}^{m}$
is Legendrian large whenever $m\leq-5$ and $1<n\leq\left\lfloor \frac{3-m}{4}\right\rfloor $. This leads to the following theorem. 

\begin{figure}

\begin{overpic}[unit=1mm,scale=1]{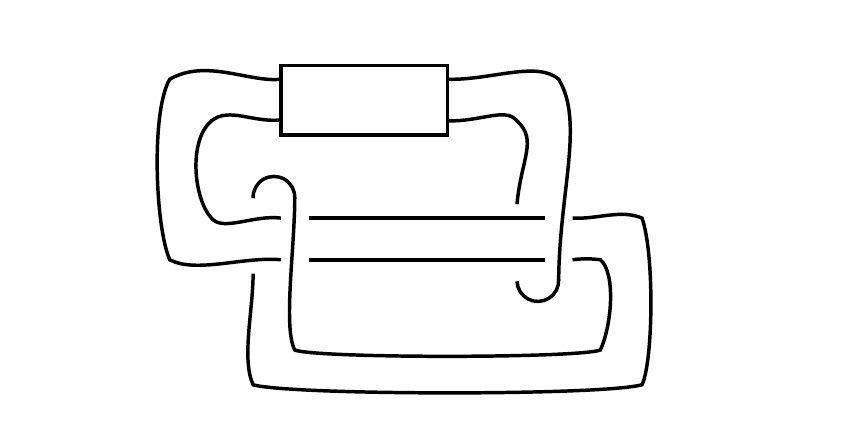}
\put(31,34){\colorbox{white}{\parbox{1.0cm}{%
    $m-1$}}}
\end{overpic}
\caption{The ribbon knots $K^{m}$. There are $m-1$ right handed full twists.}

\label{Yasui knots}

\end{figure}

\begin{thm}
The knots $K^{m}$ in Figure~\ref{Yasui knots} with $m\leq-5$,
are not uniformly thick in $\left(S^{3},\xi_{std}\right)$, in particular,
there are solid tori $T$ representing $K^{m}$ such that $slope\left(\varGamma_{\partial T}\right)>\overline{tb}\left(K^{m}\right)$
and $\xi\mid_{T}$ is tight, but virtually overtwisted. \label{Main Thm}
\end{thm}

\begin{rem}
Previously, there were no known examples of $\mathcal{K}$ in $\left(S^{3},\xi_{std}\right)$
with $w\left(\mathcal{K}\right)>\overline{tb}\left(\mathcal{K}\right)$,
except for the unknot. These are also the first examples of solid
tori in $\left(S^{3},\xi_{std}\right)$ with virtually overtwisted
contact structures. 
\end{rem}

It would be interesting to know what $w\left(K^{m}\right)$ is, and
what the possible non-thickenable tori in the knot type of $K^{m}$
are. We have the following partial result, following from Theorem~\ref{Main Thm}
and its proof.
\begin{prop}
For $m\leq-5$, the knots $K^{m}$ in Figure~\ref{Yasui knots} have
$w\left(K^{m}\right)\geq-\frac{1}{\left\lfloor \frac{3-m}{4}\right\rfloor }$.
\label{ctWidthEstimate}
\end{prop}

The origin of the examples in Theorem~\ref{Main Thm} come from an
interesting connection between contact structures and the famous cabling
conjecture first observed in \cite{key-10}. Lidman-Sivek showed
that for a knot $K$ with $\overline{tb}\left(K\right)>0$,
Legendrian surgery on $K$ (i.e. $\left(tb\left(K\right)-1\right)$-surgery)
never yields a reducible manifold. They conjectured that this might
be true with no condition on $\overline{tb}\left(K\right)$. This
is equivalent to the following conjecture for any $K$ in $S^{3}$.
\begin{conjecture}
For a Legendrian representative in the knot type $L\in\mathcal{L}\left(K_{p,q}\right)$,
$tb\left(L\right)\leq pq$.\label{conj1}
\end{conjecture}

If $tb\left(L\right)>pq$ for such an $L$, then there exists $L'$
with $tb\left(L'\right)=pq+1$ (we can always stabilize to achieve
this). Legendrian surgery on this $L'$ would then yield a reducible
manifold. In \cite{key-6}, Yasui gave some interesting examples
of ribbon knots which we will denote $K^{m}$, shown in Figure~\ref{Yasui knots}.
In what follows, we will
be concerned with integers $m<0$. 
\begin{thm}
(Yasui, \cite{key-6}) There exist infinitely many Legendrian knots in $\left(S^{3},\xi_{std}\right)$, Figure~\ref{Yasui knots}, 
each of which yields a reducible 3-manifold by a Legendrian surgery
in the standard tight contact structure. Furthermore, $K$ can be
chosen so that the surgery coefficient is arbitrarily less than $\overline{tb}\left(K\right)$. 
\end{thm}

Yasui shows that for infinitely many pairs of integers $m,n\in\mathbb{Z}$ with $m\leq-5$,
Legendrian surgery on the cables $K_{n,-1}^{m}$ yields a reducible
manifold. This shows Lidman-Sivek's conjecture to be false, and
stands in contrast with Theorem~\ref{cables<pq} of Etnyre-Honda. 

We can now easily see that $K^{m}$, Figure~\ref{Yasui knots}, does
not have the uniform thickness property. 
\begin{thm}
For integers $m\leq-5$, the ribbon knots $K^{m}$ are not uniformly
thick.\label{counterexamples}
\end{thm}

The interesting features of how $K^{m}$ fails to be uniformly thick
given in Theorem~\ref{Main Thm} require much more work. 
\begin{proof}
In \cite{key-6}, Yasui shows that for integers $n\leq\frac{3-m}{4}$,
the cables $K_{n,-1}^{m}$ have the property that $\overline{tb}\left(K_{n,-1}^{m}\right)=-1$.
But by Theorem~\ref{cables<pq}, if $K^{m}$ is uniformly thick,
then we must have that $\overline{tb}\left(K_{n,-1}^{m}\right)\leq-n$.
So for any $m\leq-5$ and any $1<n\leq\left\lfloor \frac{3-m}{4}\right\rfloor $
we arrive at a contradiction. 
\end{proof}
Theorem~\ref{counterexamples} can be used to address the following
question.

\begin{conjecture}
If $K\subset S^{3}$ is fibered, then $K$ is uniformly thick if and
only if $\xi_{K}\neq\xi_{Std}$, where $\xi_{K}$ is the contact structure
induced by an open book decomposition of $K$.\label{conjecture}
\end{conjecture}

Building on our above work, Hyunki Min \cite{key-14} recognized that
the $K^{m}$ are counterexamples. Min showed that the $K^{m}$ are
all fibered. We also know that they are slice and non-strongly-quasipositive,
which implies that $\xi_{K}\neq\xi_{Std}$ by a result of Matthew
Hedden \cite{key-12}. Theorem~\ref{counterexamples} tells us that
$K^{m}$ are not uniformly thick however, and so at least one direction
of this conjecture is false. The other direction remains an interesting
open question.

\subsection{Ribbon knots and Legendrian large cable examples}

Yasui's examples are all ribbon knots with Legendrian large cables,
and can be generalized to other families of ribbon knots. We first
observe a folk result that any ribbon knot can be described in a simple way.
\begin{thm}
\label{unknot theorem }Suppose $K\subset S^{3}$ is an arbitrary
ribbon knot with $n\in\mathbb{N}$ ribbon singularities. Then there
is an algorithm to construct a 2-handlebody for $D^{4}$ having $n-1$
or less 1-2 handle canceling pairs such that there is an unknot $U$
in the boundary of the 1-sub-handlebody which, after attaching the
2-handles, is isotopic to $K$.
\end{thm}

A representation of a ribbon knot $K$ as in Theorem~\ref{unknot theorem }
will be called a \textit{handlebody picture for $K$}. The proof of
Theorem~\ref{unknot theorem } will be given in Section~$3$. Figure~\ref{example repeated}
gives an example ribbon knot and its image after running the algorithm. 
\begin{thm}
Given an arbitrary ribbon knot $K$, we can associate to it a handlebody
picture. If it is possible to Legendrian realize the attaching circles
of the 2-handles so that the handle attachments are Stein (i.e. framings
are all $tb-1$), and also Legendrian realize $K$ so that $tb\left(K\right)=-1$,
then $K$ is a Legendrian ribbon knot that bounds a Lagrangian disk
in $\left(B^{4},\omega_{Std}\right)$. 
\end{thm}

\begin{figure}

\begin{overpic}[unit=1mm,scale=1]{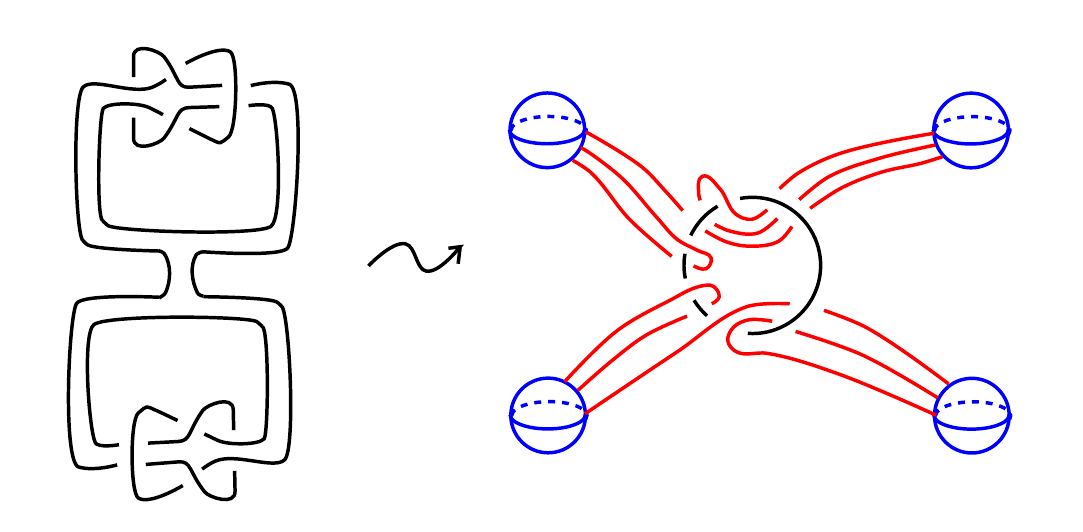}
\put(85,25){\colorbox{white}{\parbox{0.3cm}{%
    $K$}}}
\put(0,30){\colorbox{white}{\parbox{0.3cm}{%
    $K$}}}
\end{overpic}
\caption{An example ribbon knot before running the algorithm in Theorem~\ref{unknot theorem } (left), and after
running the algorithm (right).}
\label{example repeated}

\end{figure}

\begin{proof}
Given a handlebody picture for $K$, there is an unknot $U$ in the
boundary of the 1-sub-handlebody which, by hypothesis, can be realized
with $tb\left(U\right)=-1$. Such an unknot bounds a Lagrangian disk
in the 1-sub-handlebody. Since the 2-handles are attached disjointly
from this disk, $K$ bounds a Lagrangian disk after they are attached,
that is, $K$ bounds a Lagrangian disk in $\left(B^{4},\omega_{Std}\right)$. 
\end{proof}
Conway-Etnyre-Tosun \cite{key-11} make use of this fact to describe
when contact surgery on a knot in $\left(S^{3},\xi_{Std}\right)$
preserves symplectic fillability. 

\begin{figure}

\begin{overpic}[unit=1mm,scale=1]{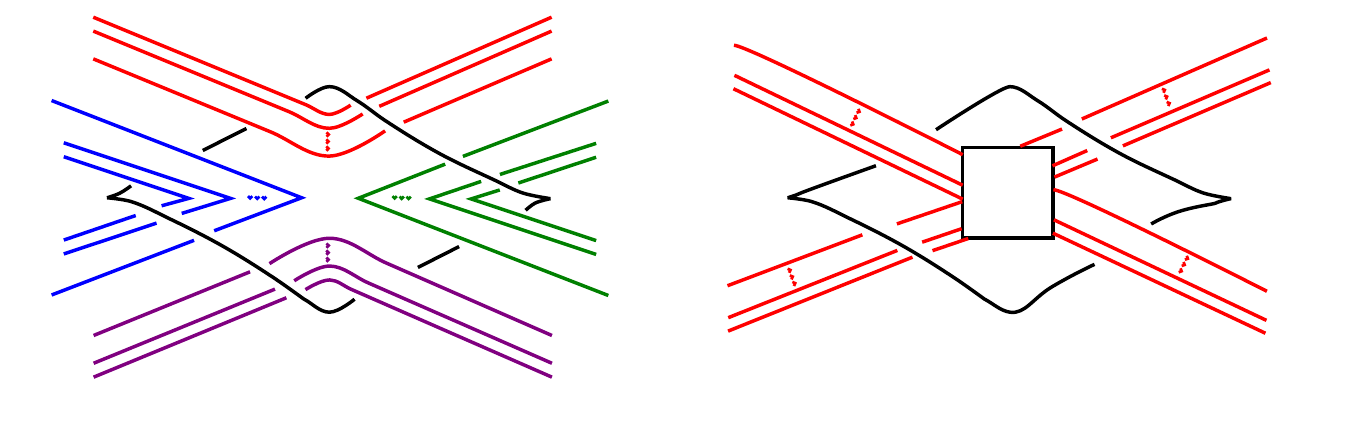}
\put(37,40){\colorbox{white}{\parbox{0.3cm}{%
    \color{myred}$\mathcal{N}$}}}
\put(37,8){\colorbox{white}{\parbox{0.3cm}{%
    \color{mypurple}$\mathcal{S}$}}}
\put(62,28){\colorbox{white}{\parbox{0.3cm}{%
    \color{mygreen}$\mathcal{E}$}}}
\put(-2,28){\colorbox{white}{\parbox{0.3cm}{%
    \color{myblue}$\mathcal{W}$}}}
\put(131,38){\colorbox{white}{\parbox{0.3cm}{%
    \color{myred}$\mathcal{NE}$}}}
\put(63,38){\colorbox{white}{\parbox{0.3cm}{%
    \color{myred}$\mathcal{NW}$}}}
\put(63,12){\colorbox{white}{\parbox{0.3cm}{%
    \color{myred}$\mathcal{SW}$}}}
\put(131,12){\colorbox{white}{\parbox{0.3cm}{%
    \color{myred}$\mathcal{SE}$}}}
\put(100,24){\colorbox{white}{\parbox{0.3cm}{%
    $\mathcal{T}$}}}
\put(30,3){\colorbox{white}{\parbox{0.3cm}{%
    $(a)$}}}
\put(99,3){\colorbox{white}{\parbox{0.3cm}{%
    $(b)$}}}
\end{overpic}
\caption{Possible examples of knots with Legendrian large cables. The ellipses
are meant to indicate a finite number of strands bundled as shown,
while $\mathcal{T}$ is an arbitrary Legendrian tangle. }
\label{generalized Yasui examples}

\end{figure}
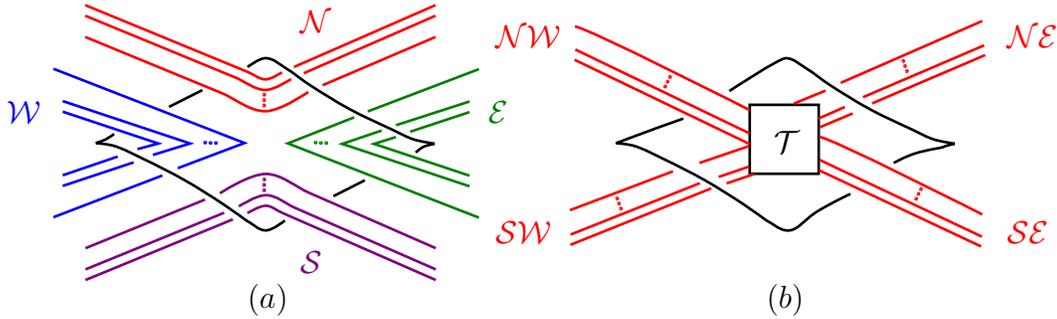

\begin{cor}
Given an arbitrary ribbon knot $K$, we can associate to it a handlebody
picture. If it is possible to Legendrian realize the attaching circles
of the 2-handles so that the handle attachments are Stein, Legendrian
realize $K$ so that $tb\left(K\right)=-1$, and also arrange the
local picture of $K$ to be as in Figure~\ref{generalized Yasui examples}~(a),
then $K$ has Legendrian large cables.
\label{Yasui examples corollary}
\end{cor}

\begin{proof}
The proof is exactly the same as the proof of Yasui's Theorem~1.3
(\cite{key-6}, pp 7-13), when there are only strands of type $\mathcal{N}$,
since everything in the arguments can be done locally. The rest of
the cases follow by Legendrian isotopy of Figure~\ref{generalized Yasui examples}~(a).
For example, we can change all strands of type $\mathcal{S}$ into
strands of type $\mathcal{N}$ by the Legendrian isotopy shown in
Figure~\ref{Legend isotopy 1}. We can also change all strands of types $\mathcal{E}$
and $\mathcal{W}$ into strands of type $\mathcal{N}$ by even easier
isotopies.
\end{proof}

\begin{figure}
\begin{overpic}[unit=1mm,scale=1]{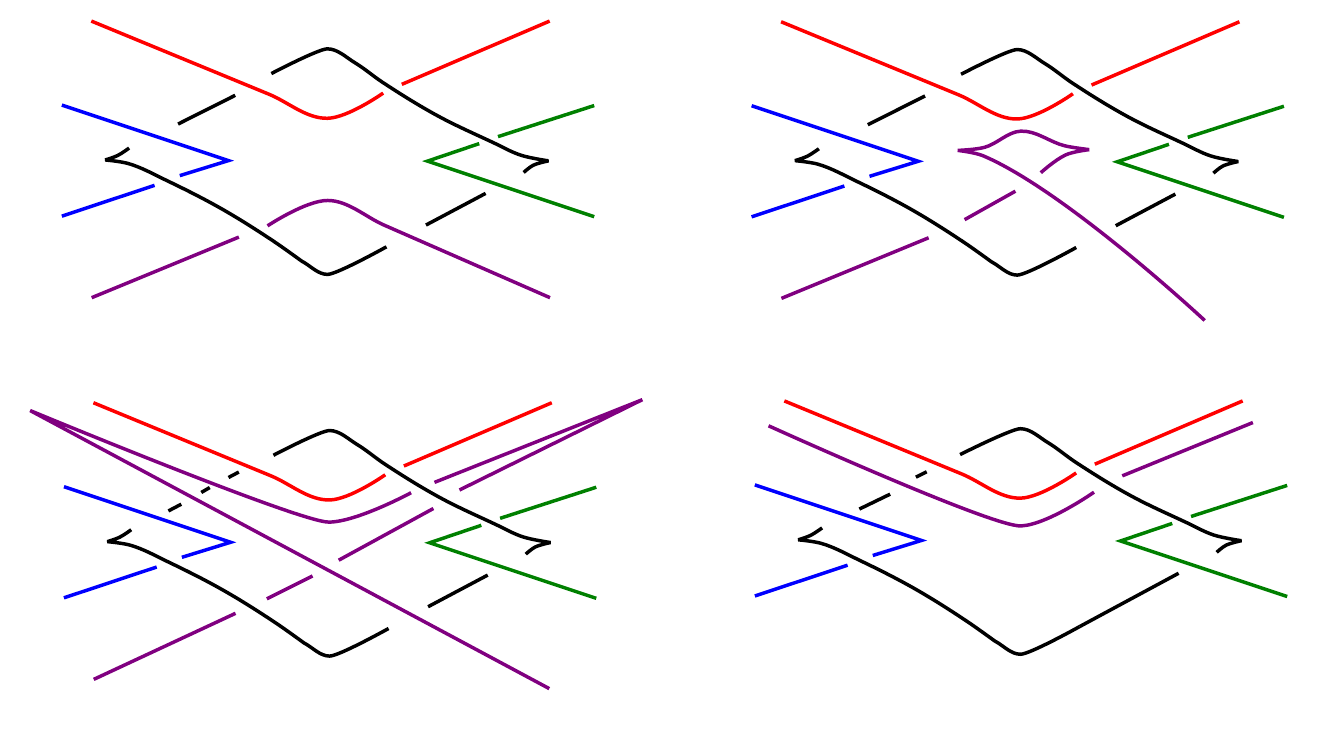}
\put(30,40){\colorbox{white}{\parbox{0.3cm}{%
    $(a)$}}}
\put(99,40){\colorbox{white}{\parbox{0.3cm}{%
    $(b)$}}}
\put(30,0){\colorbox{white}{\parbox{0.3cm}{%
    $(c)$}}}
\put(99,0){\colorbox{white}{\parbox{0.3cm}{%
    $(d)$}}}
\end{overpic}
\caption{Shows the steps in a Legendrian isotopy to change strands of type
$\mathcal{S}$ into strands of type $\mathcal{N}$.}
\label{Legend isotopy 1}
\end{figure}
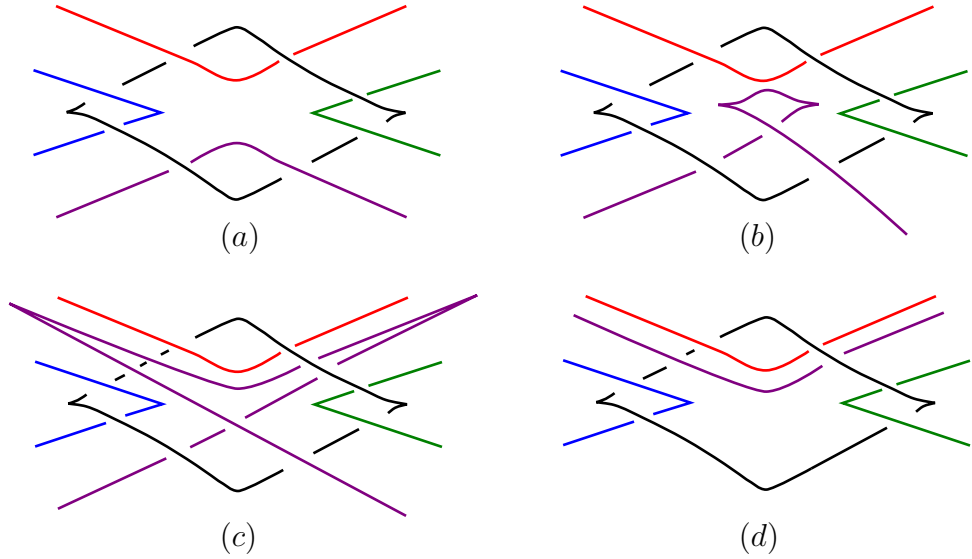

\begin{rem}
If the framings of the 2-handles allow stabilizations, then there
are more examples. Given an arbitrary ribbon knot $K$, we can associate
to it a handlebody picture. If it is possible to Legendrian realize
the attaching circles of the 2-handles so that the handle attachments
are Stein, Legendrian realize $K$ so that $tb\left(K\right)=-1$,
arrange the local picture of $K$ to be as in Figure~\ref{generalized Yasui examples}~(a),
and arrange that there is a stabilization on each of the strands of
at least one group of strands $\mathcal{NE}$, $\mathcal{NW}$, $\mathcal{SE}$,
or $\mathcal{SW}$, then $K$ has Legendrian large cables. This is
true since we can isotop the stabilizations to have the form of Figure~\ref{Legend isotopy 2}~(a),
Legendrian isotop the tangle $\mathcal{T}$ off to the side as shown
in Figure~\ref{Legend isotopy 2}~(b), and then apply Corollary~\ref{Yasui examples corollary}.
\end{rem}

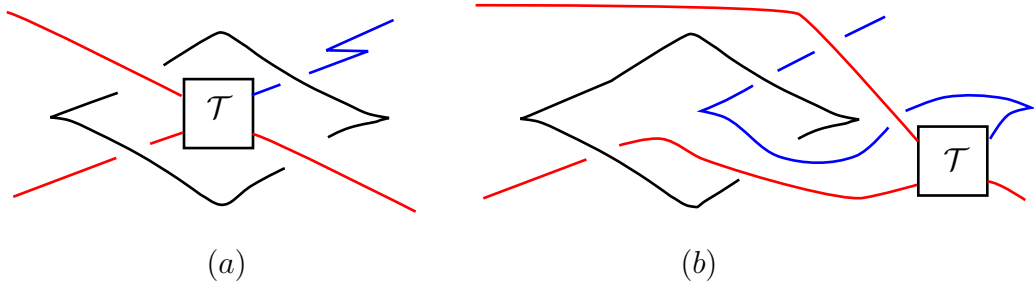
\begin{figure}
\begin{overpic}[unit=1mm,scale=1]{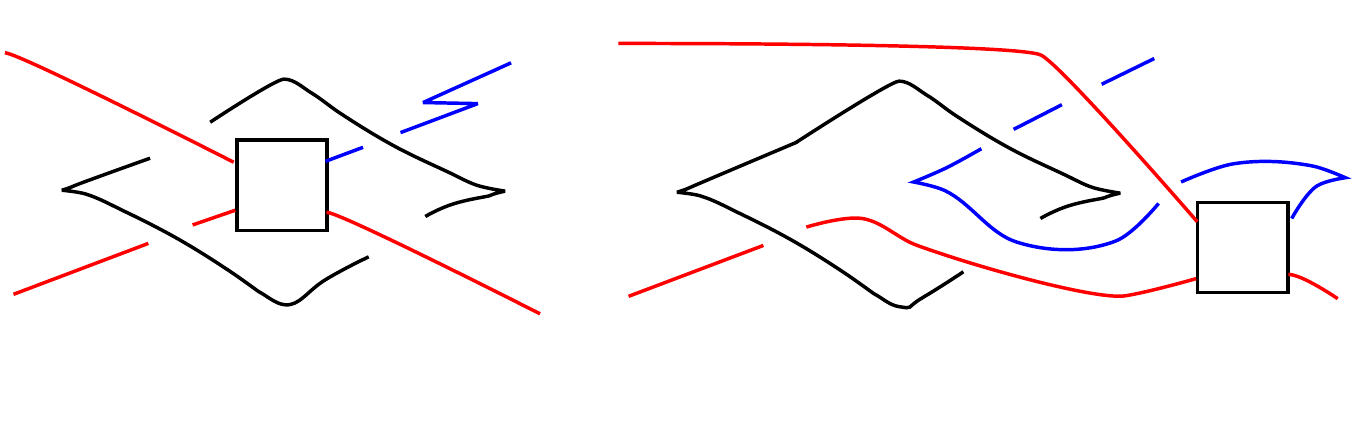}
\put(26,5){\colorbox{white}{\parbox{0.3cm}{%
    $(a)$}}}
\put(89,5){\colorbox{white}{\parbox{0.3cm}{%
    $(b)$}}}
\put(26,26){\colorbox{white}{\parbox{0.3cm}{%
    $\mathcal{T}$}}}
\put(124,19){\colorbox{white}{\parbox{0.3cm}{%
    $\mathcal{T}$}}}
\end{overpic}
\caption{Shows a Legendrian isotopy of the tangle $\mathcal{T}$. In this example,
strands of type $\mathcal{NE}$ are assumed to have stabilizations.}
\label{Legend isotopy 2}
\end{figure}

\subsubsection*{Acknowledgments: The author would like to express profound gratitude
to his advisor John B. Etnyre for his patience, encouragement, and
many helpful comments and suggestions, without which this paper would
not have been possible. He would also like to extend thanks to James Conway for making him aware of Conjecture~\ref{conjecture}, and to Sudipta Kolay, Hyunki Min, Surena Hozoori, and Peter Lambert Cole for many useful and productive conversations.  }

\section{Background}

We will assume that the reader is familiar with Legendrian knots and
basic convex surface theory. Some excellent sources for this material
are \cite{key-5,key-8,key-7,key-2}. We will need to understand the
twisting of a contact structure along a Legendrian curve with respect
to two different framings. Suppose we are given a solid torus $S\subset\left(S^{3},\xi\right)$
with convex boundary which represents the knot $K$. This just means
that $S=D^{2}\times S^{1}$ and $K=\left\{ pt\right\} \times S^{1}$ for some point in $int\left(D^{2}\right)$. Further suppose that we
are given a Legendrian $\left(p,q\right)$ curve $L$ in $S$. Since
$L$ is null-homologous in $S^{3}$, there is a well defined framing
on $L$ given by any Seifert surface $\Sigma$, and measuring the
twisting of $\xi$ along $L$ with respect to this framing gives us
$tw\left(L;\Sigma\right)=tb\left(L\right)$, that is, the Thurston-Bennequin
invariant of $L$. We can also find a boundary parallel torus $T^{2}\subset S$
containing $L$, and measure the twisting of $\xi$ along $L$ with
respect to the framing coming from $T^{2}$. We will denote this twisting
by $tw\left(L;\partial S\right)$. The relationship between these
twistings is given by the expression \cite{key-1}
\[
tw\left(L;\partial S\right)+pq=tb\left(L\right).
\]

Consider a contact structure $\xi$ on $T^{2}\times I$ with convex boundary, let $T_{1},T_{2}$
be its two torus boundary components, and assume without loss of generality
that $s_{1}=slope\left(\varGamma_{T_{1}}\right)\leq slope\left(\varGamma_{T_{2}}\right)=s_{2}$, where $\varGamma_{S}$ denotes the dividing curves on a convex surface $S$.
Then we will say that $\xi$
is \textit{minimally twisting} if every convex, boundary parallel
torus $S\subset T^{2}\times I$ has $s_{1}\leq slope\left(\varGamma_{S}\right)\leq s_{2}$.
This is the same notion of minimal twisting that Honda defined in
\cite{key-5}. We will also need to make use of his basic slices to
decompose $T^{2}\times I$ into layers. Using the same notation as
above, we will call $\left(T^{2}\times I,\xi\right)$ a \textit{basic
slice} if
\begin{enumerate}
\item $\xi$ is tight, and minimally twisting,
\item $T_{i}$ are convex, and $\#\varGamma_{T_{i}}=2$,
\item $s_{i}$ form an integral basis for $\mathbb{Z}^{2}$.
\end{enumerate}
Honda showed that, up to isotopy fixing the boundary, there are exactly
two tight contact structures on a basic slice, distinguished by their
relative Euler classes in $H^{2}\left(T^{2}\times I,\partial\left(T^{2}\times I\right);\mathbb{Z}\right)$.

The Farey tessellation, Figure~\ref{Farey Tessellation}, gives a
convenient way to describe curves on $T^{2}$.
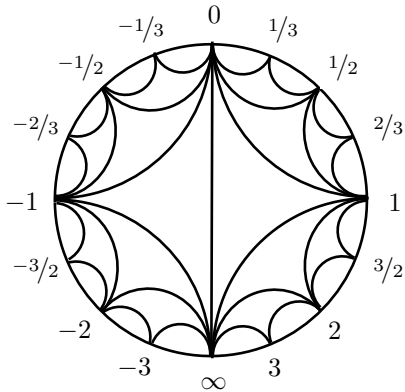
\begin{figure}
\begin{overpic}[unit=1mm,scale=1]{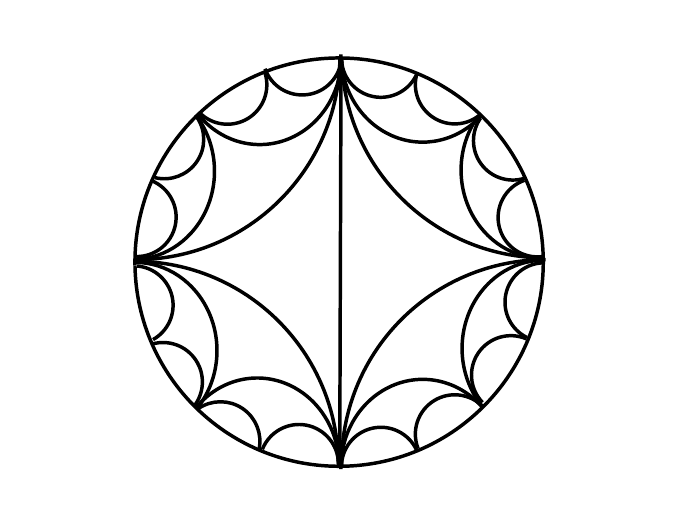}
\put(6,26){\colorbox{white}{\parbox{0.3cm}{%
    \footnotesize$-1$}}}
\put(57,26){\colorbox{white}{\parbox{0.3cm}{%
    \footnotesize$1$}}}
\put(33,51){\colorbox{white}{\parbox{0.3cm}{%
    \footnotesize$0$}}}
\put(32,2){\colorbox{white}{\parbox{0.3cm}{%
    \footnotesize$\infty$}}}
\put(21,49){\colorbox{white}{\parbox{0.3cm}{%
    \footnotesize$\nicefrac{-1}{3}$}}}
\put(21,4){\colorbox{white}{\parbox{0.3cm}{%
    \footnotesize$-3$}}}
\put(13,44){\colorbox{white}{\parbox{0.3cm}{%
    \footnotesize$\nicefrac{-1}{2}$}}}
\put(13,9){\colorbox{white}{\parbox{0.3cm}{%
    \footnotesize$-2$}}}
\put(7,36){\colorbox{white}{\parbox{0.3cm}{%
    \footnotesize$\nicefrac{-2}{3}$}}}
\put(7,17){\colorbox{white}{\parbox{0.3cm}{%
    \footnotesize$\nicefrac{-3}{2}$}}}
\put(41,49){\colorbox{white}{\parbox{0.3cm}{%
    \footnotesize$\nicefrac{1}{3}$}}}
\put(41,4){\colorbox{white}{\parbox{0.3cm}{%
    \footnotesize$3$}}}
\put(49,44){\colorbox{white}{\parbox{0.3cm}{%
    \footnotesize$\nicefrac{1}{2}$}}}
\put(49,9){\colorbox{white}{\parbox{0.3cm}{%
    \footnotesize$2$}}}
\put(55,36){\colorbox{white}{\parbox{0.3cm}{%
    \footnotesize$\nicefrac{2}{3}$}}}
\put(55,17){\colorbox{white}{\parbox{0.3cm}{%
    \footnotesize$\nicefrac{3}{2}$}}}
\end{overpic}
\caption{Farey Tessellation}
\label{Farey Tessellation}
\end{figure}
To construct the eastern half of the Farey tessellation, first label
the north pole by $0=\frac{0}{1}$, the south pole by $\infty=\frac{1}{0}$,
and connect them by an edge (by edge, we mean a hyperbolic geodesic). Next, label the eastern most point that
is midway between $0$ and $\infty$ by $1=\frac{1}{1}$, as shown
in Figure~\ref{Farey Tessellation}. Connect $1$ by edges to $0$
and $\infty$. For rational numbers on the tessellation with the same sign, we can define
an addition on the Farey tessellation by $\frac{a}{b}+\frac{c}{d}=\frac{a+c}{b+d}$,
locate $\frac{a+c}{b+d}$ midway between $\frac{a}{b}$ and $\frac{c}{d}$,
and connect $\frac{a+c}{b+d}$ by edges with $\frac{a}{b}$ and $\frac{c}{d}$
respectively. Thus we can fill in the rest of the positive side of
the Farey tessellation by iterating this addition. Notice that, if
$\frac{a}{b},\frac{c}{d}$ are assumed to be an integral basis for
$\mathbb{Z}^{2}$, then both $\begin{vmatrix}a & a+c\\
b & b+d
\end{vmatrix}=ad-bc=\begin{vmatrix}a & c\\
b & d
\end{vmatrix}=\pm1$ and similarly, $\begin{vmatrix}a+c & c\\
b+d & d
\end{vmatrix}=\pm1$, so any two points connected by an edge are an integral basis for
$\mathbb{Z}^{2}$. Also notice that, given two positive rational numbers
$\frac{a}{b}>\frac{c}{d}$, there are exactly two other points with
edges to both $\frac{a}{b}$ and $\frac{c}{d}$, namely $\frac{a+c}{b+d}$
and $\frac{a-c}{b-d}$. 

To construct the western (negative) half of the Farey tessellation,
first relabel the north pole by $0=\frac{0}{-1}$. Next, label the
western most point that is midway between $0$ and $\infty$ by $-1=\frac{1}{-1}$,
as shown in Figure~\ref{Farey Tessellation}. Connect $-1$ by edges
to $0$ and $\infty$. Now using the same addition we defined above,
we can iteratively build up the negative side of our Farey tessellation.
Notice that the only point which was labeled twice was the north
pole, which is now given by $\frac{0}{\pm1}$. 

For any two points $p_{1}$ and $p_{2}$ on the Farey tessellation,
we define the interval $\left[p_{1},p_{2}\right]$ to be the set of
all points encountered starting from $p_{1}$ and moving clockwise
around the tessellation until reaching $p_{2}$. Given a clockwise
sequence of three points connected by edges $p_{1}$, $p_{2}$ and
$p_{3}$ on the Farey tessellation, we say that a jump from $p_{2}$
to $p_{3}$ is \textit{half maximal} if $p_{3}$ is the half way point
of the maximum possible clockwise jump one could make in the interval
$\left(p_{2},p_{1}\right)$. We will consider only clockwise
paths in the Farey tessellation, where a path is a sequence of jumps
along edges. We call a path between two points $s_{1},s_{2}\in\mathbb{Q}$
a \textit{continued fraction block} if, after the first jump, every
jump is half maximal. Notice that, by construction, a path that is
a continued fraction block cannot be shortened. We will also need
to consider decorated paths (i.e. paths for which each jump gets a
``$+$'' or ``$-$''). We can define an equivalence relation ``$\sim$''
on decorated paths in the Farey tessellation which says that any two
paths with the same endpoints and which differ only by shuffling of
signs within continued fraction blocks are in the same class. The following result, due to Honda \cite{key-5}, and in a different terminology Giroux \cite{key-8}, describes a relationship between
contact structures on $T^{2}\times I$ and minimal decorated paths
in the Farey tessellation. Given a manifold $M$ and a multicurve $\varGamma$ in $\partial M$, let 
$Tight\left(M,\varGamma\right)$ denote the set of isotopy classes of tight contact structures on $M$ 
with convex boundary, such that $\varGamma$ is a set of dividing curves for $\partial M$. Similarly, given 
$T^{2}\times I$ with boundary $T_{1}\sqcup T_{2}$, and two multicurves $\varGamma_{i}$ on $T_{i}$, 
let $Tight\left(T^{2}\times I,T_{1}\cup T_{2}\right)$ denote the set of tight, minimally twisting contact structures 
on $T^{2}\times I$ with convex boundary, such that $\varGamma_{i}$ is a set of dividing curves for $T_{i}$.
\begin{thm}
(Honda, \cite{key-5}) Given $T^{2}\times I$ with boundary
$T_{1}\sqcup T_{2}$, and two multicurves $\varGamma_{i}$ on $T_{i}$ with $\#\varGamma_{i}=2$
such that $s_{1}=slope\left(\varGamma_{1}\right)\leq slope\left(\varGamma_{2}\right)=s_{2}$,
then
\[
Tight\left(T^{2}\times I,\varGamma_{1}\cup\varGamma_{2}\right)\nicefrac{\longleftrightarrow\begin{Bmatrix}minimal\;decorated\;paths\\
from\;s_{1}to\;s_{2}
\end{Bmatrix}}{\sim}.
\]
\label{classification of T2xI}
\end{thm}

Given $T^{2}\times I$ with a two component multicurve on each of its two torus boundary components, and with boundary slopes
$s_{1},s_{2}\in\mathbb{Q}$, then any decorated path starting from
$s_{1}$ and ending at $s_{2}$ describes a contact structure on $T^{2}\times I$.
Each jump in the path describes a basic slice, and therefore has two
possible contact structures distinguished by the relative Euler class.
We then get $T^{2}\times I$ by concatenating together these basic
slices. For more details, see \cite{key-5}. It follows from Theorem~\ref{classification of T2xI}
that within any continued fraction block, shuffling the signs of the
jumps results in isotopic contact structures.

Suppose we have a decorated path which can be shortened, see Figure~\ref{consistent shortening}.
It follows from Honda's gluing theorem that if the two jumps which
are being combined into a single jump have different signs, then the
contact structure on $T^{2}\times I$ described by this path is overtwisted.
If the signs agree, then the contact structure will be tight. For
this reason, we say that a shortening is \textit{consistent} if the
signs of the smaller jumps agree, and make the following theorem owing
to Honda.
\begin{thm}
Given a decorated path in the Farey tessellation from $s_{1}$ to
$s_{2}$, the contact structure on $T^{2}\times I$ with convex boundary
$T_{1}\sqcup T_{2}$, $\#\varGamma_{T_{i}}=2$, and $s_{1}=slope\left(\varGamma_{T_{1}}\right),\;slope\left(\varGamma_{T_{2}}\right)=s_{2}$
described by this path is tight if and only if every shortening is
consistent.\label{owing to Honda}
\end{thm}

To classify the tight contact structures on solid tori, we will consider
a slightly different type of path. Let a \textit{truncated path} be
a decorated path, as defined above, with the sign of the first jump
omitted from consideration. In other words, the first jump is not
decorated. Suppose we have $S^{1}\times D^{2}$ with a two component multicurve on its torus boundary, and with boundary
slope $s_{2}\in\mathbb{Q}$. If the meridian of $\partial\left(S^{1}\times D^{2}\right)$
has slope $s_{1}\in\mathbb{Q}$, then we have the following classification. 
Given $S^{1}\times D^{2}$ with boundary $T$, and a multicurve $\varGamma$
on $T$, let $Tight\left(S^{1}\times D^{2},\varGamma\right)$ denote
the set of isotopy classes of tight, minimally twisting contact structures
on $S^{1}\times D^{2}$ with convex boundary, such that $\varGamma$
is a set of dividing curves for $T$.

\begin{figure}

\begin{overpic}[unit=1mm,scale=1]{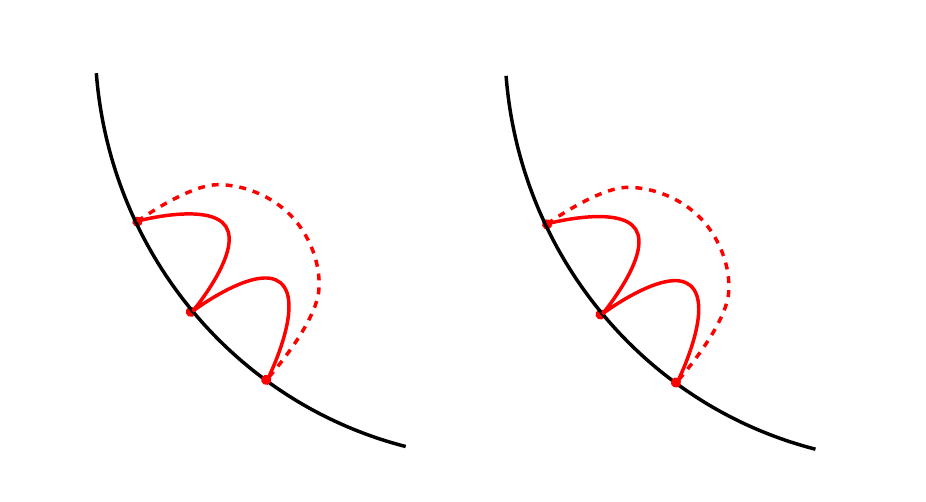}
\put(24,27){\colorbox{white}{\parbox{0.1cm}{%
    \color{myred}\tiny$+$}}}
\put(30,31){\colorbox{white}{\parbox{0.1cm}{%
    \color{myred}\tiny$+$}}}
\put(27,24){\colorbox{white}{\parbox{0.1cm}{%
    \color{myred}\tiny$+$}}}
\put(65,27){\colorbox{white}{\parbox{0.1cm}{%
    \color{myred}\tiny$+$}}}
\put(69,24){\colorbox{white}{\parbox{0.1cm}{%
    \color{myred}\tiny$-$}}}
\end{overpic}
\caption{On the left, a consistent shortening, while on the right a shortening
which is not consistent. }
\label{consistent shortening}

\end{figure}
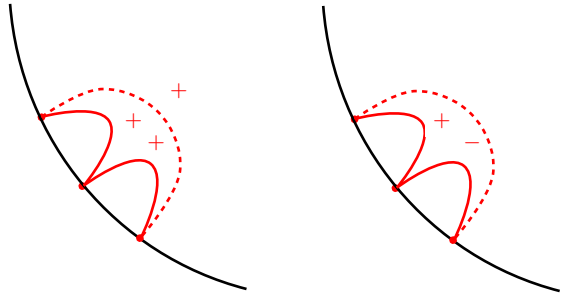

\begin{thm}
(Honda, \cite{key-5}) Given $S^{1}\times D^{2}$with boundary $T$,
and a multicurve $\varGamma$ on $T$ with $\#\varGamma=2$ such that
$s_{2}=slope\left(\varGamma\right)$, and $s_{1}=slope\left(\mu\right)$,
where $\mu$ is a meridional curve for $T$, then 
\[
Tight\left(S^{1}\times D^{2},\varGamma\right)\nicefrac{\longleftrightarrow\begin{Bmatrix}minimal\;truncated\;paths\\
from\;s_{1}\;to\;s_{2}
\end{Bmatrix}}{\sim}.
\]
\end{thm}

\begin{thm}
(Honda, \cite{key-5}) (1) Given $T^{2}\times I$ with boundary $T_{1}\sqcup T_{2}$,
and two multicurves $\varGamma_{i}$ on $T_{i}$ with $\#\varGamma_{i}=2$
such that $s_{1}=slope\left(\varGamma_{1}\right)\leq slope\left(\varGamma_{2}\right)=s_{2}$,
there are exactly two tight contact structures on $T^{2}\times I$,
and these contact structures are universally tight. The paths describing
these two structures are the same, one decorated entirely by ``$+$'',
and the other decorated entirely by ``$-$''.

(2) Given $S^{1}\times D^{2}$with boundary $T$, and a multicurve
$\varGamma$ on $T$ with $\#\varGamma=2$ such that $s_{2}=slope\left(\varGamma\right)$,
and $s_{1}=slope\left(\mu\right)$, where $\mu$ is a meridional curve
for $T$, then, if $s_{1}\cdot s_{2}\neq\pm1$, there are exactly
two tight contact structures on $S^{1}\times D^{2}$, and these contact
structure is universally tight. The paths describing these two structures
are the same, one decorated entirely by ``$+$'', and the other
decorated entirely by ``$-$''. If $s_{1}\cdot s_{2}=\pm1$, then
there exists a unique tight contact structure on $S^{1}\times D^{2}$,
and this contact structure is universally tight. \label{all same signs theorem}
\end{thm}

It follows from Theorem~\ref{all same signs theorem} that if we
have a path with a mixture of signs, then the contact structure described
by this path on either $T^{2}\times I$, or on $S^{1}\times D^{2}$,
must be virtually overtwisted.

\section{Cables in Solid Tori}

In this section, we will give the proof of Theorems~\ref{MainMainThm},
~\ref{MainMainThm2}, and ~\ref{Main Thm}. We would like to record
and make use of the following result.
\begin{thm}
\label{convex torus}
(Etnyre-Honda, \cite{key-1}) Any cable in a standard neighborhood
of a Legendrian knot can be put on a convex torus.
\end{thm}

\begin{prop}
If $\xi$ is a universally tight contact structure on a solid torus
$S$ with convex boundary, then any Legendrian $\left(p,q\right)$
knot $L\subset S$ has $tw\left(L;\partial S\right)\leq0$.\label{virtually overtwisted improved}
\end{prop}

We delay the proof of Proposition~\ref{virtually overtwisted improved}
to the end of this section, but use it here to give proofs of our
main theorems stated in the introduction.

\subsubsection*{Proof of Theorem~\ref{MainMainThm}}

If $K$ has Legendrian large cables, then there exists $L\in\mathcal{L}\left(K_{p,q}\right)$
such that $tb\left(L\right)>pq$. Take a solid torus $S$ representing
$K$ and containing $L$ as a $\left(p,q\right)$ curve. Perturb $S$
to have convex boundary. By hypothesis $tw\left(L;\partial S\right)>0$,
so by Proposition~\ref{virtually overtwisted improved}, $\left.\xi \right|_{S}$
must be virtually overtwisted. Suppose that it were possible to thicken
$S$ to a standard neighborhood $\widetilde{S}$ of $K$. Then $slope\left(\varGamma_{\partial\widetilde{S}}\right)\in\mathbb{Z}$
which implies by a result of Kanda \cite{key-7}, that $\left.\xi \right|_{\widetilde{S}}$
is the unique tight contact structure on $\widetilde{S}$, and moreover
that $\left.\xi \right|_{\widetilde{S}}$ is universally tight. But this is
a contradiction since $S\subset\widetilde{S}$ and $\left.\xi \right|_{S}$
is virtually overtwisted, so no such thickening exists. If $K$ were
uniformly thick, then any neighborhood of $K$ would be thickenable
to a $slope\left(\overline{tb}(K)\right)$ standard neighborhood of
$K$, which we have just seen is not possible. \textifsymbol[ifgeo]{48}

\subsubsection*{Proof of Theorem~\ref{MainMainThm2}}

By assumption, there exists $L\in\mathcal{L}\left(K_{p,q}\right)$
such that $tb\left(L\right)>pq$. Stabilize $L$ to obtain $\widetilde{L}$
such that $tb\left(\widetilde{L}\right)=pq$. There is a solid torus
$S$ representing $K$ for which $\widetilde{L}\subset\partial S$,
and as discussed at beginning of Section $2$, we see that $tw\left(\widetilde{L};\partial S\right)=0$. We can
therefore $C^{0}$ perturb a collar neighborhood $N$ of $\widetilde{L}$
in $\partial S$ to be convex, and then $C^{\infty}$ perturb $\partial S\setminus N$
to obtain a solid torus $\widetilde{S}$ representing $K$ with convex
boundary. Since $tw\left(\widetilde{L};\partial\widetilde{S}\right)=0$, and
since $slope\left(\widetilde{L}\right)=\frac{q}{p}$, we must have
that $slope\left(\varGamma_{\partial\widetilde{S}}\right)=\frac{q}{p}$,
owing to the fact that $tw\left(\widetilde{L};\widetilde{S}\right)=-\frac{1}{2}\left|\widetilde{L}\bullet\varGamma_{\partial\widetilde{S}}\right|$
where $C_{1}\bullet C_{2}$ denotes the geometric intersection number
of two curves on a torus. But $\frac{q}{p}>\overline{tb}(K)$ by assumption,
so $w\left(K\right)>\overline{tb}(K)$. \textifsymbol[ifgeo]{48}

\subsubsection*{Proof of Theorem~\ref{Main Thm}}

In \cite{key-6}, Yasui shows that for integers $n\leq\frac{3-m}{4}$,
the cables $K_{n,-1}^{m}$ have the property that $\overline{tb}\left(K_{n,-1}^{m}\right)=-1$.
So for any $m\leq-5$ and any $1<n\leq\left\lfloor \frac{3-m}{4}\right\rfloor $
we see that $K^{m}$ has Legendrian large cables $L\in\mathcal{L}\left(K_{n,-1}^{m}\right)$.
Then by Theorem~\ref{MainMainThm} $K^{m}$ is not uniformly thick
and has virtually overtwisted neighborhoods, and by Theorem~\ref{MainMainThm2}
we have that $w\left(K^{m}\right)>\overline{tb}(K^{m})$ \textifsymbol[ifgeo]{48}

\subsubsection*{Proof of Proposition~\ref{ctWidthEstimate}}

The slope of the cable $K_{n,-1}^{m}$ is $slope\left(K_{n,-1}^{m}\right)=-\frac{1}{n}$.
Whenever $n\leq\frac{3-m}{4}$, we know there exist $L\in\mathcal{L}\left(K_{n,-1}^{m}\right)$
which are Legendrian large. Stabilize $L$ to obtain $\widetilde{L}$
such that $tb\left(\widetilde{L}\right)=-n$. There is a solid torus
$S$ representing $K^{m}$ for which $\widetilde{L}\subset\partial S$,
and we have seen that $tw\left(\widetilde{L};\partial S\right)=0$. Using the
strategy of the proof of Theorem~\ref{MainMainThm2}, we can $C^{0}$
perturb a collar neighborhood $N$ of $\widetilde{L}$ in $\partial S$
to be convex, and then $C^{\infty}$ perturb $\partial S\setminus N$
to obtain a solid torus $\widetilde{S}$ representing $K^{m}$ with
convex boundary. Since $tw\left(\widetilde{L};\partial\widetilde{S}\right)=0$,
and since $slope\left(\widetilde{L}\right)=-\frac{1}{n}$, we must
have that $slope\left(\varGamma_{\partial\widetilde{S}}\right)=-\frac{1}{n}$,
and therefore that $w\left(K^{m}\right)\geq-\frac{1}{n}$. \textifsymbol[ifgeo]{48}

Now we will give a series of results leading to the proof of Proposition~\ref{virtually overtwisted improved}.

\begin{lem}
\label{lem:twisting<=00003Dpq}Let $S$ be a solid torus with convex
boundary, $\left|\varGamma_{\partial S}\right|=2$, and $slope\left(\varGamma_{\partial S}\right)\in\mathbb{Z}$
with its unique tight contact structure $\xi$, then any Legendrian
$\left(p,q\right)$ knot $L\subset S$ has $tw\left(L;\partial S\right)\leq0$.
\end{lem}

\begin{proof}
Notice that this follows immediately from Theorem~\ref{convex torus}, since $S$ is a standard neighborhood, and any Legendrian curve $L$ on a convex torus $T$ must have $tw\left(L;T\right)=tw\left(L;\partial S\right)\leq0$. Alternatively, we can reason in the following way.
Recall that Kanda \cite{key-7} showed that any solid torus with integer
slope and two dividing curves has a unique tight contact structure.
Suppose that $S$ is a solid torus with convex boundary, $\left|\varGamma_{\partial S}\right|=2$,
and $slope\left(\varGamma_{\partial S}\right)=k\in\mathbb{Z}$ with
its unique tight contact structure $\xi$, and that $L\subset S$
is a Legendrian $\left(p,q\right)$ knot. Then $S$ is a standard
neighborhood of a Legendrian core curve $K$. Any two standard neighborhoods
are contactomorphic, so we can find a neighborhood $N\subset\left(S^{3},\xi_{std}\right)$
of a Legendrian unknot $U\subset S^{3}$ with $tb\left(U\right)=-1$,
and a contactomorphism $\varphi:S\rightarrow N$ which sends $\varphi\left(K\right)=U$.
This contactomorphism sends torus knots to torus knots, so our $\left(p,q\right)$
knot $L$ is mapped to a $\left(p,q-p\left(k+1\right)\right)$ knot
$\varphi\left(L\right)$ as one can easily check. But now $\varphi\left(L\right)$
is a torus knot in $\left(S^{3},\xi_{std}\right)$, and Etnyre and
Honda have shown \cite{key-2} that $tb\left(\varphi\left(L\right)\right)\leq p\left(q-p\left(k+1\right)\right)$.
But we understand how to switch between the Seifert framing and the
framing coming from the torus $\partial N$, that is, $tw\left(\varphi\left(L\right);\partial N\right)=tb\left(\varphi\left(L\right)\right)-p\left(q-\left(k+1\right)\right)\leq0$.
This implies that $tw\left(L;\partial S\right)\leq0$, since $N$
and $S$ are contactomorphic.
\end{proof}
We can strengthen Lemma~\ref{lem:twisting<=00003Dpq} slightly by
dropping the assumption that $\left|\varGamma\right|=2$.
\begin{lem}
Let $S$ be a solid torus with convex boundary, and $slope\left(\varGamma_{\partial S}\right)\in\mathbb{Z}$
with any tight contact structure $\xi$, then any Legendrian $\left(p,q\right)$
knot $L\subset S$ has $tw\left(L;\partial S\right)\leq0$.\label{lem:Let-improved}
\end{lem}

\begin{proof}
We will show that $\left(S,\xi\right)$ will embed in a tight contact
structure $\left(\widetilde{S},\widetilde{\xi}\right)$ that satisfies the hypothesis of
Lemma~\ref{lem:twisting<=00003Dpq}, and therefore show that $tw\left(L;\partial S\right)\leq0$.
To this end, we note that we can assume $slope\left(\varGamma_{\partial S}\right)=0$
by applying a diffeomorphism to $S$. Recall \cite{key-5}, that $\xi$
is completely determined by the dividing set $\varGamma_{D}$ on a meridional
disk $D$ of $S$. We will build a model situation for $S$ in which
we can construct $\left(\widetilde{S},\widetilde{\xi}\right)$. Since $\left|\varGamma_{\partial S}\right|>2$
we see that $\left|\varGamma_{D}\right|>1$. Suppose that we have
a convex disk $D$ with an arbitrary collection of dividing curves
$\varGamma$, as in Figure~\ref{Disk with arbitrary dividing curves}. 

\begin{figure}

\includegraphics{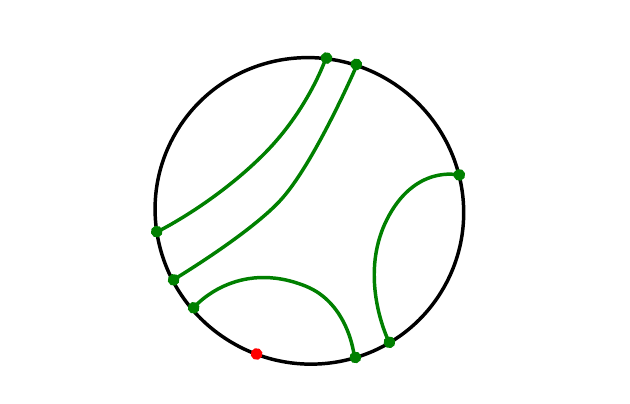}\caption{Arbitrary disk with arcs.}

\label{Disk with arbitrary dividing curves}

\end{figure}

Let $v$ be a vector field on $D$ that guides the characteristic foliation. We can
label the regions in $D\setminus\varGamma$ as either $\varSigma^{+}$
or $\varSigma^{-}$ so that no adjacent pair share the same label.
There exists an area form $\omega$ on $D$ which satisfies that $\pm div_{\omega}v>0$
on $\varSigma^{\pm}$. Assign a 1-form $\lambda=\iota_{v}\omega$, then we know from Giroux \cite{key-8} that there exists a function
$u:D\rightarrow\mathbb{R}$ such that $udt+\lambda$ gives rise to
a contact structure $\xi$ on $D\times\mathbb{R}$ that is invariant
in the $\mathbb{R}$ direction. Moreover, we know from a theorem of
Giroux that $\xi$ is tight, since there are no homotopically trivial
dividing curves. This invariance means that we can mod out by $\mathbb{Z}$
to obtain a tight contact structure on a solid torus $\nicefrac{D\times\mathbb{R}}{\mathbb{Z}}=D\times S^{1}$. 
The solid torus and contact structure we obtain in this way are contactomorphic to our original $\left(S,\xi\right)$, that is, there exists $v,\omega$ and $u:D\rightarrow\mathbb{R}$ for which this construction exactly reproduces $\left(S,\xi\right)$.

Now suppose that the number of properly embedded arcs is greater than
$1$. We would now like to reduce the number of dividing curves by
taking a larger disk containing our original $D$. So we attach an
annulus to $D$ to obtain $D_{ext}=D\cup_{\varphi}\left(S^{1}\times\left[0,1\right]\right)$
where $\varphi:S^{1}\times\left\{ 0\right\} \rightarrow\partial D$
is the gluing map. Denote the endpoints of the properly embedded arcs
as $\left\{ x_{1},\ldots,x_{2k}\right\} $. Notice that if we fix
a point on $p\in\partial D$ and move counterclockwise from $p$ along
$\partial D$, then it must happen that we encounter an $x_{i}$ followed
by an $x_{i+1}$ which are not endpoints of the same curve. If this
were not so, then there could only be one curve, which we have supposed
not to be the case. Without loss of generality, assume that these
two points are $x_{1}$ and $x_{2}$. Now connect these points by
an arc in $S^{1}\times\left[0,1\right]$. Form arcs from the remaining
points $\left\{ x_{3},\ldots,x_{2k}\right\} $ to $\partial D_{ext}$
by using $\left\{ x_{i}\right\} \times\left[0,1\right]$, as in Figure~\ref{Extended Disk}.
Notice that $D_{ext}$ has one fewer embedded arcs than $D$. So we
can iterate this procedure to obtain a disk $\widetilde{D}\supset D$
which has only $1$ properly embedded arc. Call this arc $\widetilde{\varGamma}$.
Notice that we can arrange the gluing map $\varphi$ to be smooth
and such that the extension of $\varGamma$ to $\widetilde{\varGamma}$
is smooth. We can also smoothly extend $\omega$ and $v$ to $\widetilde{D}$
so that the singular foliation on $\widetilde{D}$ guided by $v$
has $\widetilde{\varGamma}$ as a dividing curve. We can now build,
just as we did above, a contact structure $\widetilde{\xi}$ on $\widetilde{D}\times S^{1}=\widetilde{S}$
having $\widetilde{D}$ as a convex meridional disk, with convex boundary.
Since $\left|\varGamma_{\widetilde{D}}\right|=1$ we see that $tb\left(\partial\widetilde{D}\right)=-1$,
which in turn implies that $\left|\varGamma_{\partial\widetilde{S}}\right|=2$.
Notice that $\left.\widetilde{\xi}\right|_{S}=\xi$. Also notice that,
by construction, the method of reducing the number of dividing curves
on $\partial S$ yields $slope\left(\varGamma\right)=slope\left(\widetilde{\varGamma}\right)$.
Now by Lemma~\ref{lem:twisting<=00003Dpq}, any Legendrian $\left(p,q\right)$
knot $L\subset S$ has $tw\left(L;\partial S\right)\leq0$.

\begin{figure}

\includegraphics{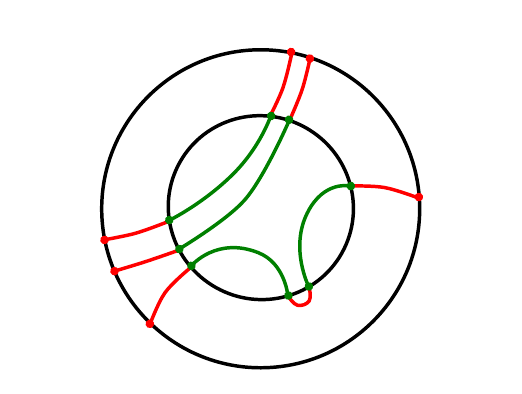}\caption{An annulus has been attached, and the number of curves has been reduced
by one.}
\label{Extended Disk}
\end{figure}
\end{proof}
\begin{lem}
If $\xi$ is a universally tight contact structure on a solid torus
$S$ with convex boundary and $\left|\varGamma_{\partial S}\right|=2$,
then any Legendrian $\left(p,q\right)$ knot $L\subset S$ has $tw\left(L;\partial S\right)\leq0$.\label{virtually overtwisted}
\end{lem}

\begin{proof}
By a diffeomorphism of $S$, we can assume that $slope\left(\varGamma_{\partial S}\right)=-\nicefrac{r}{s}$
where $-\infty\leq-\nicefrac{r}{s}\leq-1$, and that the meridional
slope is $-\infty$. Let $n=\left\lceil \nicefrac{r}{s}\right\rceil $.
Then since $\xi$ is universally tight, we know that any path in the
Farey tessellation, describing our contact structure has the property
that each jump must be decorated with the same sign by Theorem~\ref{classification of T2xI}. 
A portion of the Farey tessellation shows this in Figure~\ref{Farey tessellation}.
We can obtain a larger solid torus $\widetilde{S}\supset S$, convex,
two dividing curves, and with $slope\left(\varGamma_{\partial\widetilde{S}}\right)=-n+1$
in the following way. Take a shortest path in the Farey tessellation
from $-\nicefrac{r}{s}$ to $-n+1$, and decorate each jump with the
sign which appears in the description of the contact structure on
$S$. This describes a contact structure on $T^{2}\times I$ which extends $S$ to $\widetilde{S}$,
and since the signs are all the same we know that $\widetilde{S}$
is tight by Theorem~\ref{owing to Honda}. Moreover, we see that
$\widetilde{S}$ has integer slope giving it a unique tight contact
structure. Now we have that $tw(K;\partial S)\leq0$ by Lemma~\ref{lem:twisting<=00003Dpq}.

\begin{figure}
\begin{overpic}[unit=1mm,scale=1]{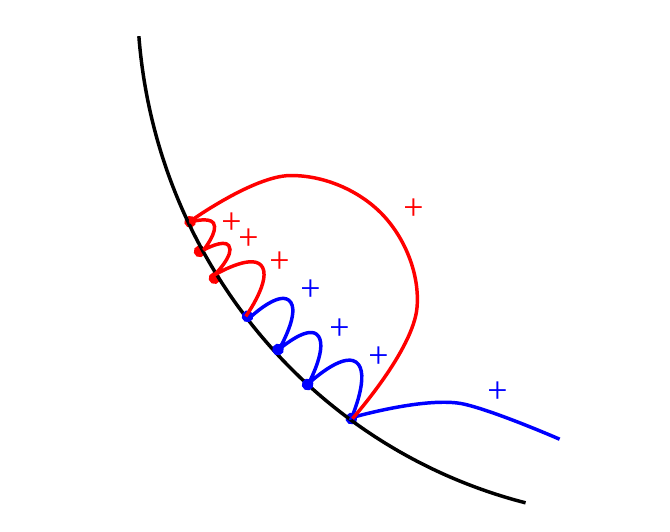}
\put(29,8){\colorbox{white}{\parbox{0.1cm}{%
    $-n$}}}
\put(3,29){\colorbox{white}{\parbox{1.2cm}{%
    $-n+1$}}}
\put(15,18){\colorbox{white}{\parbox{0.1cm}{%
    $-\nicefrac{r}{s}$}}}
\end{overpic}

\caption{Farey tessellation picture describing the contact structure on our
solid torus. The original solid torus, $S$, is shown in blue, while
the red indicates the $T^{2}\times I$ which is glued on to obtain
the larger solid torus $\widetilde{S}$. }

\label{Farey tessellation}

\end{figure}
\end{proof}
\begin{rem}
In the above proof, we are able to thicken $S$ to a larger solid
torus $\widetilde{S}\supseteq S$ with $slope\left(\varGamma_{\partial\widetilde{S}}\right)=-n+1$
because we are thinking of $S=S^{1}\times D^{2}$ abstractly as a
contact 3-manifold with convex boundary, and not embedded in any particular contact manifold. There is a shortest path
in the standard Farey tessellation picture from any negative rational
$-\nicefrac{r}{s}$ to $-n+1$ which describes our contact structure.
We are not claiming that if $S$ is a solid torus representing a knot $K\subseteq S^{3}$
it must always be thickenable in $S^{3}$, for example, Etnyre, LaFountain,
and Tosun have given examples of non-thickenable tori in \cite{key-3}.
\end{rem}

Proposition~\ref{virtually overtwisted improved} strengthens Lemma~\ref{virtually overtwisted}
slightly by dropping the assumption that $\left|\varGamma\right|=2$.

\subsubsection*{Proof of Proposition~\ref{virtually overtwisted improved} }

Suppose we are given a solid torus $S$ with convex boundary, a universally
tight contact structure $\xi$, and we have a Legendrian $\left(p,q\right)$
knot $L$ in $S$. Again, by a diffeomorphism of $S$, we can assume
that $slope\left(\varGamma_{\partial S}\right)=-\nicefrac{r}{s}$
where $-\infty\leq-\nicefrac{r}{s}\leq-1$, that the meridional slope
is $-\infty$. Let $n=\left\lceil \nicefrac{r}{s}\right\rceil $.
If $\left|\varGamma_{\partial S}\right|=2k>2$, then we can attach
a bypass to $\partial S$ along a Legendrian ruling curve to obtain
a smaller solid torus $S'\subset S$ which has $slope\left(\varGamma_{\partial S'}\right)=-\nicefrac{r}{s}$
and $\left|\varGamma_{\partial S'}\right|=2k-2$. We can repeat this
procedure until we have a solid torus $\widetilde{S}\subset S$ which
has $slope\left(\varGamma_{\partial\widetilde{S}}\right)=-\nicefrac{r}{s}$
and $\left|\varGamma_{\partial\widetilde{S}}\right|=2$. Notice that
the contact structure on $\widetilde{S}$ is just $\left.\xi\right|_{\widetilde{S}}$.
If we look at a meridional disk $D\subset S$, we know that along $\partial D$ there are $2sk$ intersection
points with $\varGamma_{\partial S}$, however there exists a slope
$\gamma$ for which, curves on $\partial S$ of slope $\gamma$ have
exactly $2k$ intersection points with $\varGamma_{\partial S}$.
For convenience, change coordinates on $S$ so that $slope\left(\gamma\right)\mapsto-\infty$
and $slope\left(\varGamma_{\partial S}\right)\mapsto0$. 
Notice that we have a $T^{2}\times I$ layer $X=S\setminus\widetilde{S}$,
and we can find a convex annulus $A$ in $X$ with Legendrian boundary of slope $\gamma$. We would like to show that the
contact structure on $X$ is completely determined by the dividing
curves on $A$. Since $\left|\varGamma_{\partial\widetilde{S}}\right|=2$,
$\left|\varGamma_{\partial S}\right|=2k$, and $slope\left(\varGamma_{\partial\widetilde{S}}\right)=slope\left(\varGamma_{\partial S}\right)=0$,
we know that the dividing curves on $A$ must have the form shown
in Figure~\ref{X}~(b) by the green arcs.

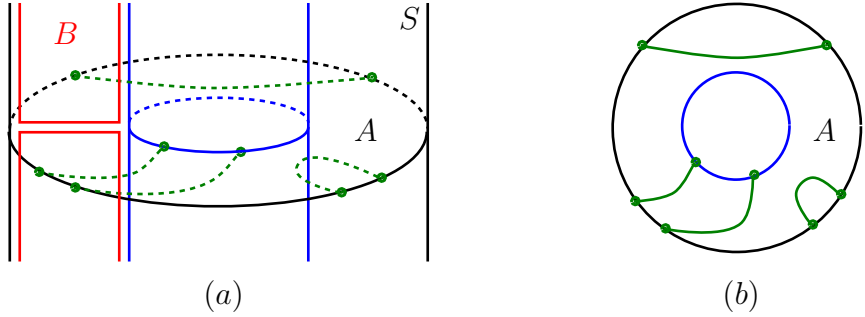
\begin{figure}

\begin{overpic}[unit=1mm,scale=1]{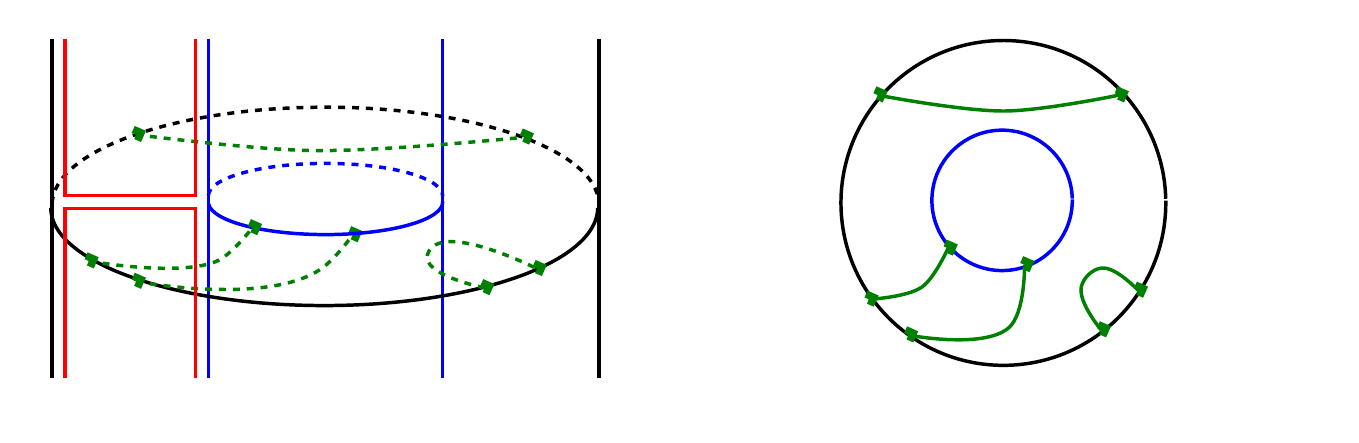}
\put(56,38){\colorbox{white}{\parbox{0.1cm}{%
    $S$}}}
\put(50,23){\colorbox{white}{\parbox{0.1cm}{%
    $A$}}}
\put(111,23){\colorbox{white}{\parbox{0.1cm}{%
    $A$}}}
\put(30,1){\colorbox{white}{\parbox{0.1cm}{%
    $(a)$}}}
\put(99,1){\colorbox{white}{\parbox{1.4cm}{%
    $(b)$}}}
\put(10,36){\colorbox{white}{\parbox{0.1cm}{%
    \color{myred}$B$}}}
\end{overpic}

\caption{On the left, $X=T^{2}\times I$, and on the right, the annulus $A$
and its dividing curves.}
\label{X}

\end{figure}
We know from Giroux \cite{key-8} that the contact structure on a
neighborhood of $A$ is determined by its dividing curves. If we cut
$X$ along $A$, and round corners, we obtain a solid torus $Y$ with convex boundary. Using
the edge rounding lemma \cite{key-5}, it is easy to see that $\left|\varGamma_{\partial Y}\right|=2$
and $slope\left(\varGamma_{\partial Y}\right)=-1$. Notice in Figure~\ref{X}~(a)
that we have a meridional disk $B$ of $Y$ which we have just seen
has $tw\left(\partial B\right)=-1$, and which we can perturb to be
convex. There is a unique choice of dividing curves on such a disk.
Finally, if we cut $Y$ along $B$ and round corners, we obtain a
$B^{3}$ with convex boundary, which has a unique tight contact structure
from work of Eliashberg \cite{key-13}. So we have seen that the contact structure
of $X$ is determined solely by the dividing curves on $A$. 

Let $v$ be a vector field on $A$ that guides the characteristic
foliation. We can label the regions in $A\setminus\varGamma$ as either
$\varSigma^{+}$ or $\varSigma^{-}$ so that no adjacent pair share
the same label. There exists an area form $\omega$ on $A$ which
satisfies that $\pm div_{\omega}v>0$ on $\varSigma^{\pm}$. Assign
a 1-form $\lambda=\iota_{v}\omega$, then we know from Giroux \cite{key-8}
that there exists a function $u:A\rightarrow\mathbb{R}$ such that
$udt+\lambda$ gives rise to a contact structure $\xi$ on $A\times\mathbb{R}$
that is invariant in the $\mathbb{R}$ direction. Moreover, we know
from a theorem of Giroux that $\xi$ is tight, since there are no
homotopically trivial dividing curves. This invariance means that
we can mod out by $\mathbb{Z}$ to obtain a tight contact structure
on $\nicefrac{A\times\mathbb{R}}{\mathbb{Z}}=T^{2}\times I$. The
$T^{2}\times I$ layer and contact structure we obtain in this way
are contactomorphic to our original $\left(X,\xi\right)$, that is,
there exists $v,\omega$ and $u:A\rightarrow\mathbb{R}$ for which
this construction exactly reproduces $\left(X,\xi\right)$.

Now observe that we can smoothly extend $A$, abstractly, by an annulus $\widehat{A}$
causing the number of dividing curves to be reduced to $2$, just as we did
with the disk in the proof of Lemma~\ref{lem:Let-improved}, see
Figure~\ref{larger annulus}. 

\begin{figure}

\begin{overpic}[unit=1mm,scale=1]{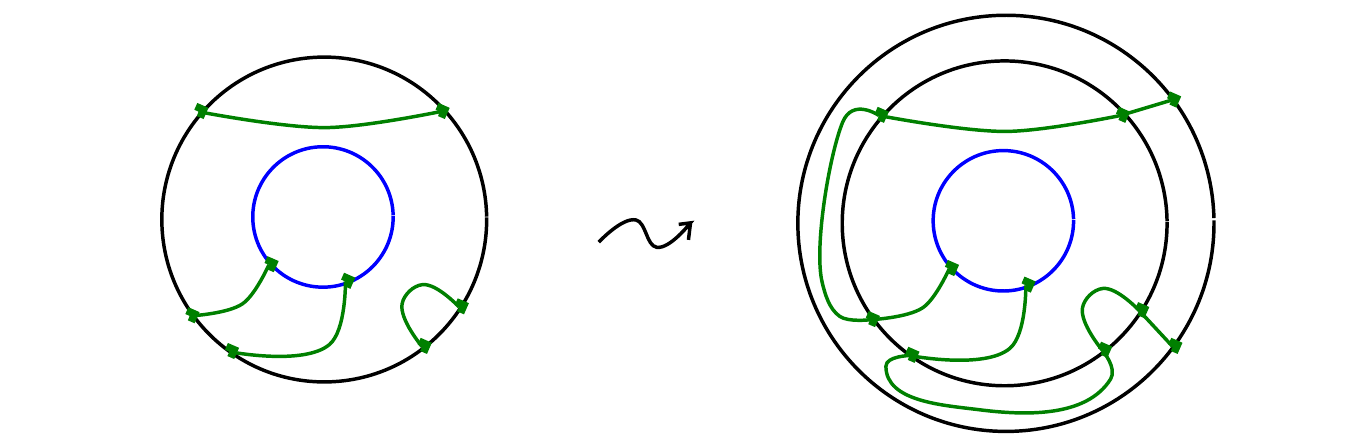}
\put(42,23){\colorbox{white}{\parbox{0.1cm}{%
    \footnotesize$A$}}}
\put(111,23){\colorbox{white}{\parbox{0.1cm}{%
    \footnotesize$A$}}}
\put(119,23){\colorbox{white}{\parbox{0.1cm}{%
    \footnotesize$\widehat{A}$}}}
\end{overpic}

\caption{Reducing the number of dividing curves on $A$ by extending with an
annulus $\widehat{A}$}
\label{larger annulus}

\end{figure}

We can arrange that the extension of $\varGamma_{A}$ to $\varGamma_{A\cup\widehat{A}}$
is smooth, and we can also smoothly extend $\omega$ and $v$ to a
neighborhood of $\widehat{A}$ so that the singular foliation on $\widehat{A}$
guided by $v$ has $\varGamma_{A\cup\widehat{A}}$ as a set of dividing
curves.  We can now build, just
as we did above, a contact structure $\widehat{\xi}$ on $\left(A\cup\widehat{A}\right)\times S^{1}=\widehat{X}$
with convex boundary. Since $\left|\varGamma_{A\cup\widehat{A}}\right|=2$
we see that $tb\left(\partial\widehat{A}\cap\partial\widehat{X}\right)=-1$,
which implies that $\left|\varGamma_{\partial\widehat{X}}\right|=2$.
Notice that $\left.\widehat{\xi}\right|_{X}=\xi$. Also notice that,
by construction, the method of reducing the number of dividing curves
on $\partial X$ yields $slope\left(\varGamma_{\partial X}\right)=slope\left(\varGamma_{\partial\widehat{X}}\right)$.
But now we have a minimally twisting $T^{2}\times I$ layer $\widehat{X}$
whose boundary tori each have two dividing curves with slope $0$.
Honda showed \cite{key-5} that there are an integers worth of tight
contact structures satisfying these boundary conditions, and that
each one is $I$-invariant. Adding the $I$-invariant thickened torus
$X\cup\widehat{X}$ to $\widetilde{S}$ we get a new solid torus with
contact structure contactomorphic to $\left.\xi\right|_{\widetilde{S}}$,
thus universally tight. Clearly $S$ is contained in this solid torus.
Now by Lemma~\ref{virtually overtwisted}, $L$ has $tw\left(L;\partial S\right)=tw\left(L;\partial S'\right)\leq0$.
\textifsymbol[ifgeo]{48}

\section{Building Ribbon Knots from Canceling Handles}

We are concerned here with ribbon knots, which we take to be the following: 

A knot $K\subset S^{3}$ is a \textit{ribbon knot} if there is an
immersed disk $\varphi:D_{K}^{2}\rightarrow S^{3}$ such that 
\begin{enumerate}
\item $\partial\varphi\left(D_{K}\right)=K$
\item all of the double points of $\varphi\left(D_{K}\right)=\widetilde{D}_{K}$
(we will use the symbol $\sim$ to denote image under $\varphi$)
occur transversely along arcs $\gamma_{i}\subset S^{3}$ whose pre-image
$\varphi^{-1}\left(\gamma_{i}\right)\subset D_{K}$ consists of exactly
two arcs. One of these, $\alpha_{i}$, must be contained entirely
in the interior, $\alpha_{i}\subset int\left(D_{K}\right)$, and the
other, $\beta_{i}$ (meant to suggest boundary), must be a properly
embedded arc in $D_{K}$ (i.e. $\partial\beta_{i}\subset\partial D_{K}$
and $int\left(\beta_{i}\right)\cap\partial D_{K}=\emptyset$). 
\end{enumerate}
An example ribbon disk and its image under $\varphi$ are shown in
Figure~\ref{ribbon disk}.

\begin{figure}[h]

\begin{overpic}[unit=1mm,scale=1]{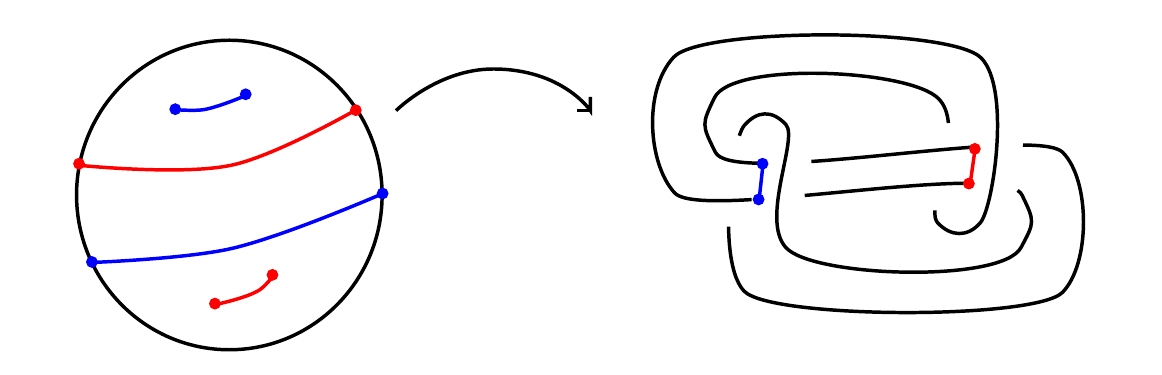}
\put(4,31){\colorbox{white}{\parbox{0.1cm}{%
    $D_{K}$}}}
\put(47,33){\colorbox{white}{\parbox{0.1cm}{%
    $\varphi$}}}
\put(102,31){\colorbox{white}{\parbox{1.4cm}{%
    $K\subset S^{3}$}}}
\put(29,9){\colorbox{white}{\parbox{0.1cm}{%
    \color{myred}$\alpha_{2}$}}}
\put(0,20){\colorbox{white}{\parbox{0.1cm}{%
    \color{myred}$\beta_{2}$}}}
\put(91,24){\colorbox{white}{\parbox{0.1cm}{%
    \color{myred}$\gamma_{2}$}}}
\put(26,27){\colorbox{white}{\parbox{0.1cm}{%
    \color{myblue}$\alpha_{1}$}}}
\put(2,9){\colorbox{white}{\parbox{0.1cm}{%
    \color{myblue}$\beta_{1}$}}}
\put(68,14){\colorbox{white}{\parbox{0.1cm}{%
    \color{myblue}$\gamma_{1}$}}}
\end{overpic}
\caption{The immersion of a ribbon disk $D_{K}$}
\label{ribbon disk}

\end{figure}
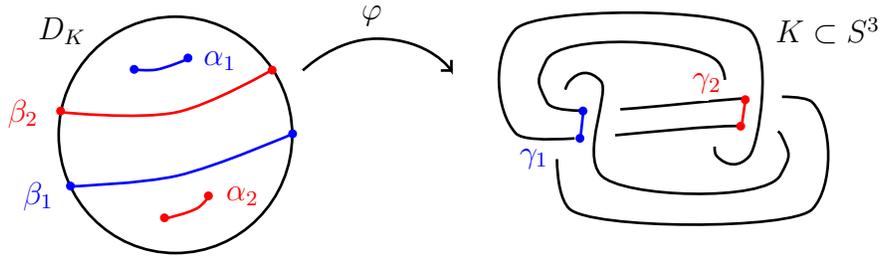

Note that by transversality, the pre-images of the $\gamma_{i}'s$
are 1-dimensional sub-manifolds of the compact manifold $D_{K}$,
so there are only finitely many ribbon singularities $\gamma_{i}$.

We want to give a construction of an arbitrary ribbon knot using 1-2
handle canceling pairs. Given any ribbon knot $K\subset S^{3}$,
it has a ribbon disk $D_{K}$ by the definition. Notice, every ribbon
singularity, $\gamma_{i}$, must appear exactly twice on the ribbon
disk, once as a properly embedded arc, and once as an arc contained
entirely in the interior of $D_{K}$. We will use a common color when
picturing these pairs. So a general ribbon disk might look something
like the one seen in Figure~\ref{general ribbon disk}.

\begin{figure}

\includegraphics{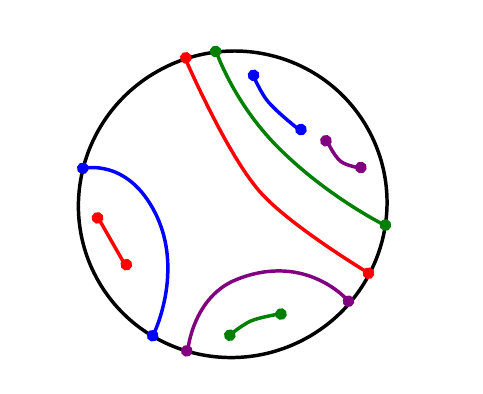}

\caption{A general ribbon disk example}
\label{general ribbon disk}

\end{figure}

We will want to make cuts, $c_{j}$, by pushing off two parallel copies
of an arc in $D_{K}$ and removing a small $\epsilon$-strip. The
result of this cut is shown in Figure~\ref{cuts}. 

\begin{figure}

\begin{overpic}[unit=1mm,scale=1]{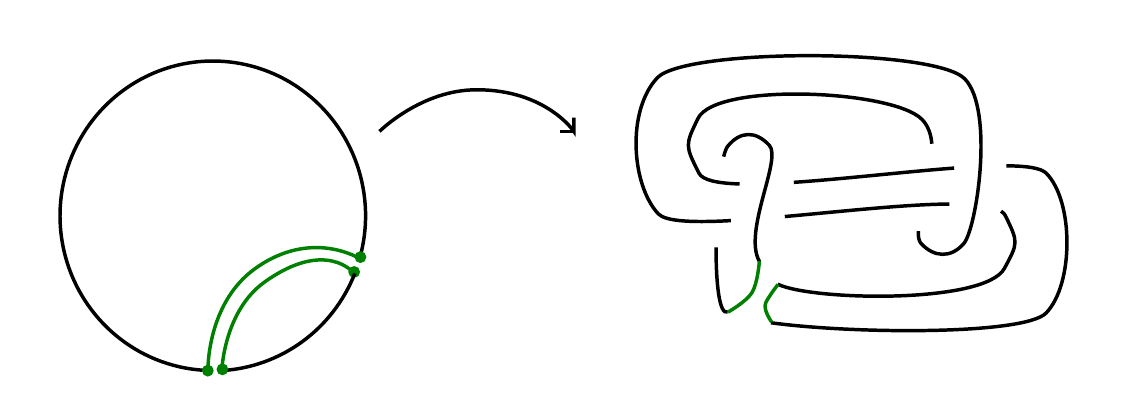}
\put(46,34){\colorbox{white}{\parbox{0.1cm}{%
    $\varphi$}}}
\put(72,5){\colorbox{white}{\parbox{0.1cm}{%
    \color{mygreen}$c_{j}$}}}
\end{overpic}
\caption{Cutting a ribbon disk along an arc $c_{j}$}
\label{cuts}

\end{figure}

We also need to set up a tool for manipulating ribbon disks and their
images. Suppose we have an arc $b\subset\partial D_{K}$ whose end-points
are the end-points of one of our $\beta's$. Further suppose that
the subdisk $D$ they bound contains no other singular points as in
Figure~\ref{subdisk}. 

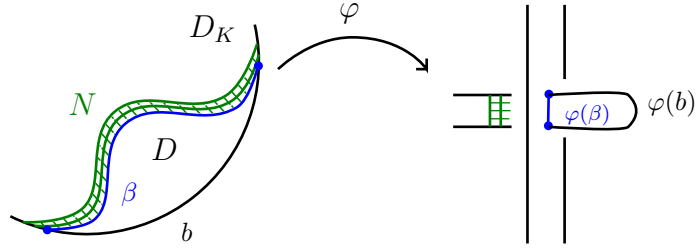
\begin{figure}

\begin{overpic}[unit=1mm,scale=1]{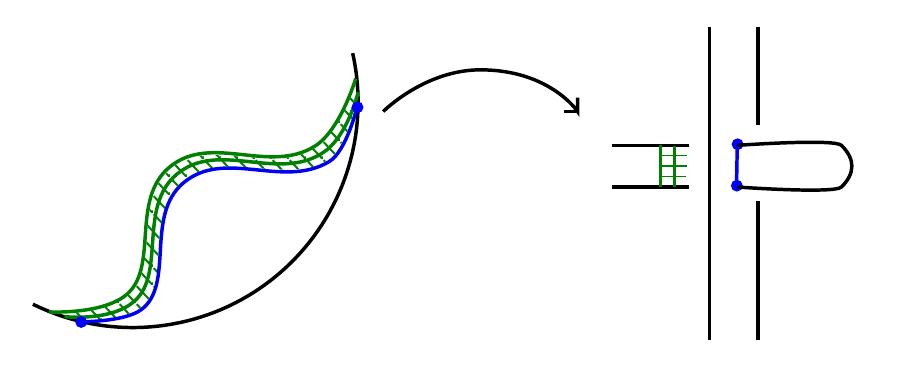}
\put(46,32){\colorbox{white}{\parbox{0.1cm}{%
    $\varphi$}}}
\put(26,30){\colorbox{white}{\parbox{0.1cm}{%
    $D_{K}$}}}
\put(21,14){\colorbox{white}{\parbox{0.1cm}{%
    $D$}}}
\put(25,3){\colorbox{white}{\parbox{0.1cm}{%
    \footnotesize$b$}}}
\put(87,20){\colorbox{white}{\parbox{0.1cm}{%
    \footnotesize$\varphi (b)$}}}
\put(10,20){\colorbox{white}{\parbox{0.1cm}{%
    \color{mygreen}$N$}}}
\put(17,8){\colorbox{white}{\parbox{0.1cm}{%
    \color{myblue}\footnotesize$\beta$}}}
\put(76,21){\colorbox{white}{\parbox[t][0.01cm][t]{0.01cm}{%
    \color{myblue}\tiny$\varphi (\beta)$}}}
\end{overpic}
\caption{A sub-disk and collar neighborhood, and its immersed image under $\varphi$.}
\label{subdisk}

\end{figure}

Let $N=I\times\left[0,\epsilon\right]$ be a collar neighborhood of
$\beta$ in $D_{K}$ such that $\left(t,0\right)=\beta$. We can form
a new disk $D_{\epsilon}=D\cup N$ with boundary
\[
\partial D_{\epsilon}=b\cup\left(0,s\right)\cup\left(1,s\right)\cup\left(t,\epsilon\right)
\]
and notice that $int\left(\beta\right)\subset int\left(D_{\epsilon}\right)$.
By choosing $\epsilon>0$ sufficiently small, we can assume that $\widetilde{D}_{\epsilon}$
is embedded. Then we can see that $D_{\epsilon}$ guides an isotopy,
supported in a small neighborhood of $D_{K}$, taking $b$ to $\left(t,\epsilon\right)$
so that the disk $\overline{D_{K}-D_{\epsilon}}=D_{K}'$ does not
contain $\beta$. We will refer to such a move as a \textit{disk slide}.
Figure~\ref{disk slide} shows a typical disk slide. 

\begin{figure}

\begin{overpic}[unit=1mm,scale=1]{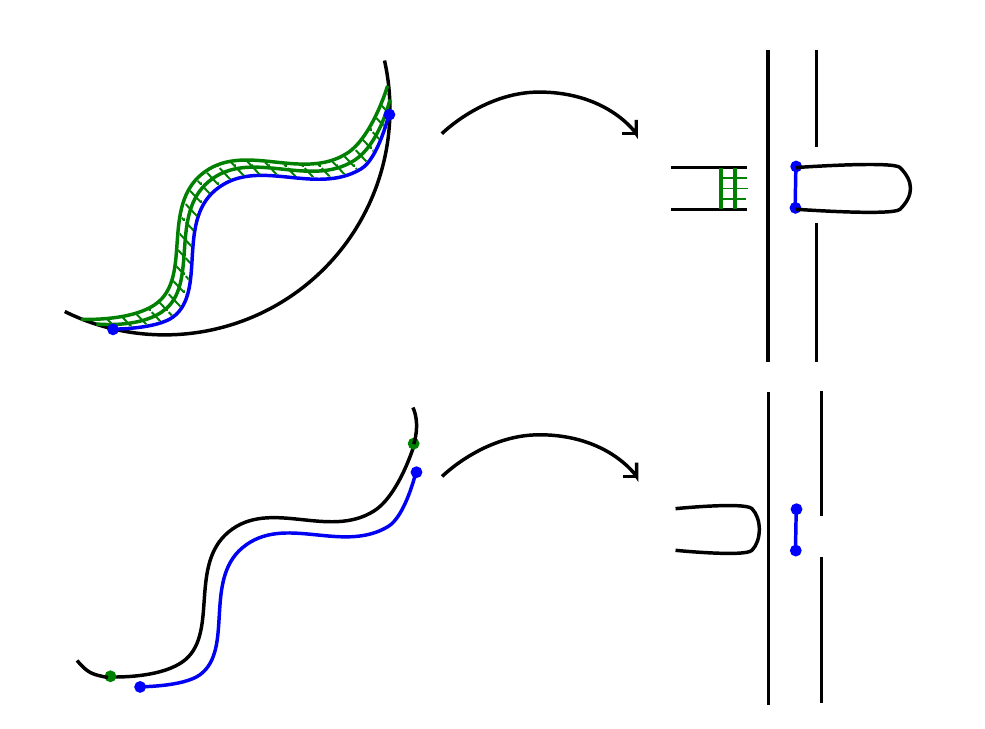}
\put(52,67){\colorbox{white}{\parbox{0.1cm}{%
    $\varphi$}}}
\put(52,32){\colorbox{white}{\parbox{0.1cm}{%
    $\varphi$}}}
\put(26,67){\colorbox{white}{\parbox{0.1cm}{%
    $D_{K}$}}}
\put(26,30){\colorbox{white}{\parbox{0.1cm}{%
    $D_{K}$}}}
\put(23,49){\colorbox{white}{\parbox{0.1cm}{%
    $D$}}}
\put(20,43){\colorbox{white}{\parbox{0.1cm}{%
    \color{myblue}\footnotesize$\beta$}}}
\put(82,56){\colorbox{white}{\parbox[t][0.01cm][t]{0.01cm}{%
    \color{myblue}\tiny$\varphi (\beta)$}}}
\put(23,8){\colorbox{white}{\parbox{0.1cm}{%
    \color{myblue}\footnotesize$\beta$}}}
\put(84,19){\colorbox{white}{\parbox{0.1cm}{%
    \color{myblue}\tiny$\varphi (\beta)$}}}
\end{overpic}
\caption{An illustration of a disk slide.}
\label{disk slide}

\end{figure}
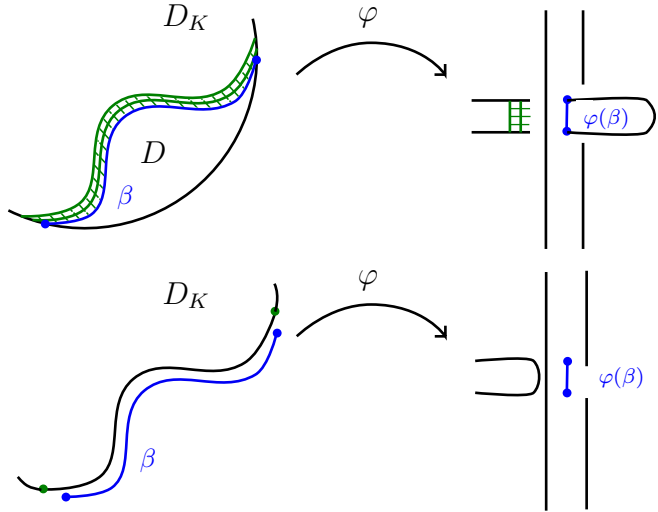

\begin{thm}
\label{thm:cutting to get disks}Given an arbitrary ribbon knot $K\subset S^{3}$
with $n\in\mathbb{N}$ ribbon singularities, $\gamma_{i}$, we can
make $n-1$ or less cuts, $c_{j}$, so that what remains of $K$ is
an unlink, and what remains of $\widetilde{D}_{K}$ is, after $n$
or less disk slides, embedded. That is, it is a collection of disjoint
disks. 
\end{thm}

To prove this we will need the following.
\begin{lem}
\label{removing a singularity}Given a ribbon knot $K$ with $n$
ribbon singularities, if we can find a subdisk $D\subset D_{K}$ such
that 
\[
\partial D=\left(an\;arc\;in\;\partial D_{K}\right)\cup\beta_{i}
\]
for one of our properly embedded arcs, $\beta_{i}$ , and $int\left(D\right)$
is disjoint from all $\alpha's$ and $\beta's$, then a disk slide
gives an isotopy of $K$ supported in a small neighborhood of $D$
so that the new slicing disk $\widetilde{D}_{K}'$ has $n-1$ ribbon
singularities. 
\end{lem}

\begin{proof}
For reference, let $b_{i}=\partial D\cap\partial D_{K}$ so that $\partial D=b_{i}\cup\beta_{i}$.
Also, let $N=I\times\left[0,\epsilon\right]$ be a collar neighborhood
of $\beta_{i}$ in $D_{K}$ such that $\left(t,0\right)=\beta_{i}$
similar to the one shown in Figure~\ref{subdisk}.

Then we can define a new subdisk $D_{\epsilon}=D\cup N$ with boundary
$\partial D_{\epsilon}=b_{i}\cup\left(0,s\right)\cup\left(1,s\right)\cup\left(t,\epsilon\right)$
and notice that $int\left(\beta_{i}\right)\subset int\left(D_{\epsilon}\right)$.
By choosing $\epsilon>0$ sufficiently small, we can assume that $\widetilde{D}_{\epsilon}$
is embedded. Then there is a disk slide taking $b_{i}$ to $\left(t,\epsilon\right)$
so that the disk $\overline{D_{K}-D_{\epsilon}}=D_{K}'$ does not
contain $\beta_{i}$. But then it also cannot contain $\alpha_{i}$,
since the pre-images of singularities occur in pairs, and hence the
singularity, $\gamma_{i}$, has been eliminated. We also have that
the resulting knot, $\partial\widetilde{D}_{K}'$, is isotopic to
$K$.
\end{proof}
Notice that Lemma~\ref{removing a singularity} says that if we see
a boundary parallel arc in $D_{K}$ with no other singular points
between that arc and some portion of $\partial D_{K}$, that we can
eliminate that arc and its interior partner from the picture by an
isotopy of $K$. Now back to our general picture and the proof of
Theorem~\ref{thm:cutting to get disks}. 

\subsubsection*{Proof of Theorem~\ref{thm:cutting to get disks}}

We will assume that our ribbon disk is reduced in the sense that,
if it were possible to simplify with a disk slide, then we have done
so already. We will consider Figure~\ref{general ribbon disk} as
our prototypical ribbon disk, and recall the convention that for each
singularity $\gamma_{i}$, $\varphi^{-1}\left(\gamma_{i}\right)$
consists of $\alpha_{i}\cup\beta_{i}$ with $\beta_{i}$ properly
embedded. Given an arbitrary ribbon knot $K\subset S^{3}$ with $n\in\mathbb{N}$
ribbon singularities, $\gamma_{i}$, and ribbon disk $\varphi:D_{K}\rightarrow S^{3}$,
there will always be an ``outermost'' properly embedded arc, $\beta_{i}$.
This means that in some subdisk, $D$, whose boundary is $\beta_{i}$
together with an arc $b_{i}\subset\partial D_{K}$, there are only
interior singular points, $\alpha_{j}$, and no other properly embedded
arcs. \textcompwordmark Figure~\ref{cutting ribbon disk}~(a) shows
one such case. 

\begin{figure}

\begin{overpic}[unit=1mm,scale=1]{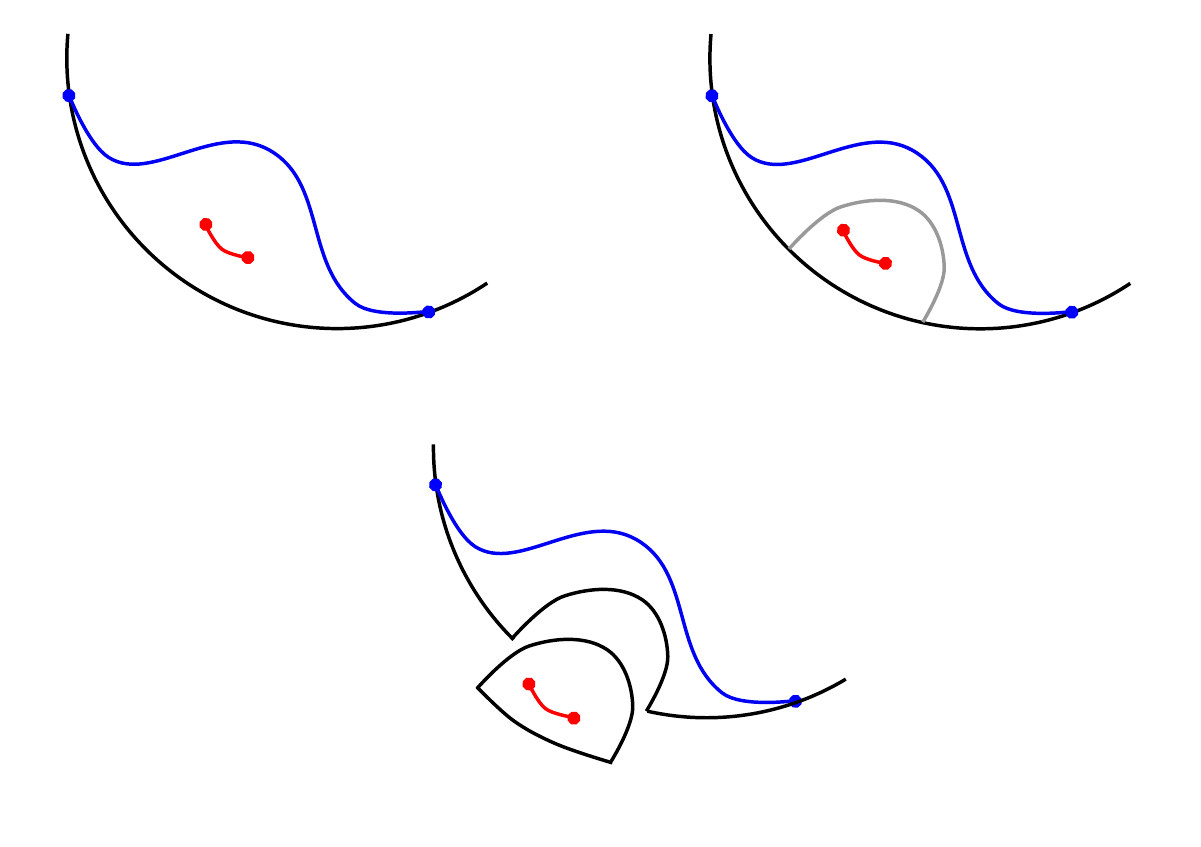}
\put(10,79){\colorbox{white}{\parbox{0.1cm}{%
    $D_{K}$}}}
\put(75,79){\colorbox{white}{\parbox{0.1cm}{%
    $D_{K}$}}}
\put(13,65){\colorbox{white}{\parbox{0.1cm}{%
    $D$}}}
\put(47,37){\colorbox{white}{\parbox{0.1cm}{%
    $D_{K}'$}}}
\put(55,17){\colorbox{white}{\parbox{0.1cm}{%
    $D'$}}}
\put(50,26){\colorbox{white}{\parbox{0.1cm}{%
    $S$}}}
\put(76,58){\colorbox{white}{\parbox{0.1cm}{%
    \color{mygray}\footnotesize$c$}}}
\put(12,56){\colorbox{white}{\parbox{0.1cm}{%
    \footnotesize$b_{i}$}}}
\put(23,74){\colorbox{white}{\parbox{0.1cm}{%
    \color{myblue}\footnotesize$\beta_{i}$}}}
\put(88,74){\colorbox{white}{\parbox{0.1cm}{%
    \color{myblue}\footnotesize$\beta_{i}$}}}
\put(60,35){\colorbox{white}{\parbox{0.1cm}{%
    \color{myblue}\footnotesize$\beta_{i}$}}}
\put(26,58){\colorbox{white}{\parbox{0.1cm}{%
    \color{myred}\tiny$\alpha_{j}$}}}
\put(91,58){\colorbox{white}{\parbox{0.1cm}{%
    \color{myred}\tiny$\alpha_{j}$}}}
\put(59,13){\colorbox{white}{\parbox{0.1cm}{%
    \color{myred}\tiny$\alpha_{j}$}}}
\put(22,47){\colorbox{white}{\parbox{0.1cm}{%
    $(a)$}}}
\put(88,47){\colorbox{white}{\parbox{0.1cm}{%
    $(b)$}}}
\put(59,0){\colorbox{white}{\parbox{0.1cm}{%
    $(c)$}}}
\end{overpic}
\caption{Cutting a ribbon disk}
\label{cutting ribbon disk}

\end{figure}
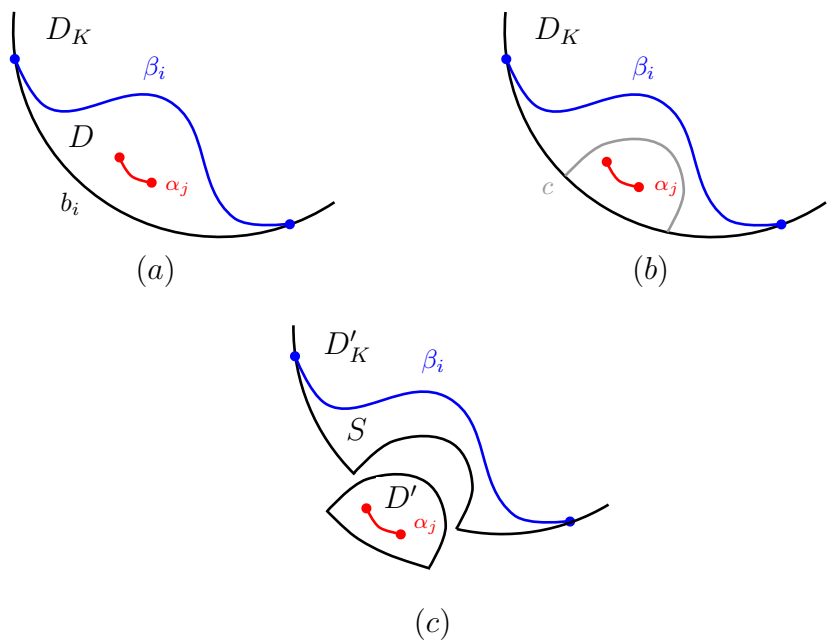

Let $c$ be a properly embedded arc in $D\subset D_{K}$ such that
$c$ cuts $D$ into $D'\cup S$ with $\beta_{i}\subset S$ and $D'$
containing all arcs $\alpha_{j}\subset D$. We may cut $D_{K}$ along
$c$ so that $\varphi$ is defined on $D_{K}'=\overline{D_{K}-D'}$
and $D'$, and after a small isotopy of $\varphi\mid_{D'}$ we have
that $\varphi\left(D'\right)$ and $\varphi\left(D_{K}'\right)$ are
disjoint as pictured in Figure~\ref{cutting ribbon disk} (c). Then
a disk slide eliminates $\beta_{i}$ by Lemma~\ref{removing a singularity}.
Notice that when we eliminate a particular $\beta_{i}$ using a disk
slide, that automatically eliminates the corresponding $\alpha_{i}$
since they occur in pairs. Also notice, each cut eliminates at least
one $\beta_{i}$, but could allow for the removal of more than one.
But after at most $n-1$ cuts we have at most one $\beta_{j}$ and
its corresponding $\alpha_{j}$. Since $\beta_{j}$ cuts the disk
it sits on into two components, one of them contains no $\alpha$
curves, see Figure~\ref{final iteration}, and so $\beta_{j}$ can
be removed with no further cuts. Thus we never need to make the $n^{th}$
cut since this last $\beta$ curve may be eliminated by a disk slide
without making a cut. Then the image under $\varphi$ is now $n$
embedded disks whose boundary is an unlink. \textifsymbol[ifgeo]{48}

\begin{figure}

\begin{overpic}[unit=1mm,scale=1]{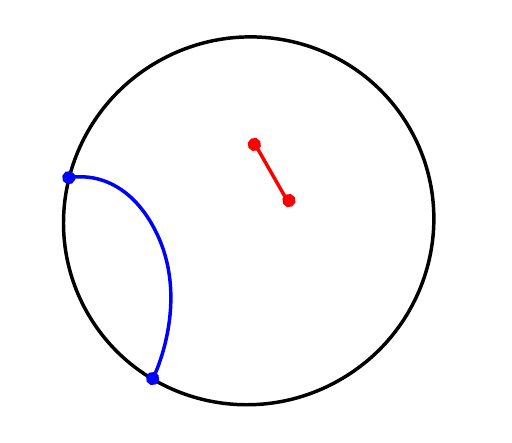}
\put(14,33){\colorbox{white}{\parbox{0.1cm}{%
    $D_{K}'$}}}
\put(30,20){\colorbox{white}{\parbox{0.1cm}{%
    \color{myred}$\alpha_{j}$}}}
\put(10,14){\colorbox{white}{\parbox{0.1cm}{%
    \color{myblue}$\beta_{j}$}}}
\end{overpic}
\caption{Final iteration}
\label{final iteration}
\end{figure}

We remark that this gives an upper bound on the number of cuts needed,
but there are certainly cases where this number is not optimal as
the following example shows.
\begin{example}
Consider the ribbon knot in Figure~\ref{single cut}.

\begin{figure}

\begin{overpic}[unit=1mm,scale=1]{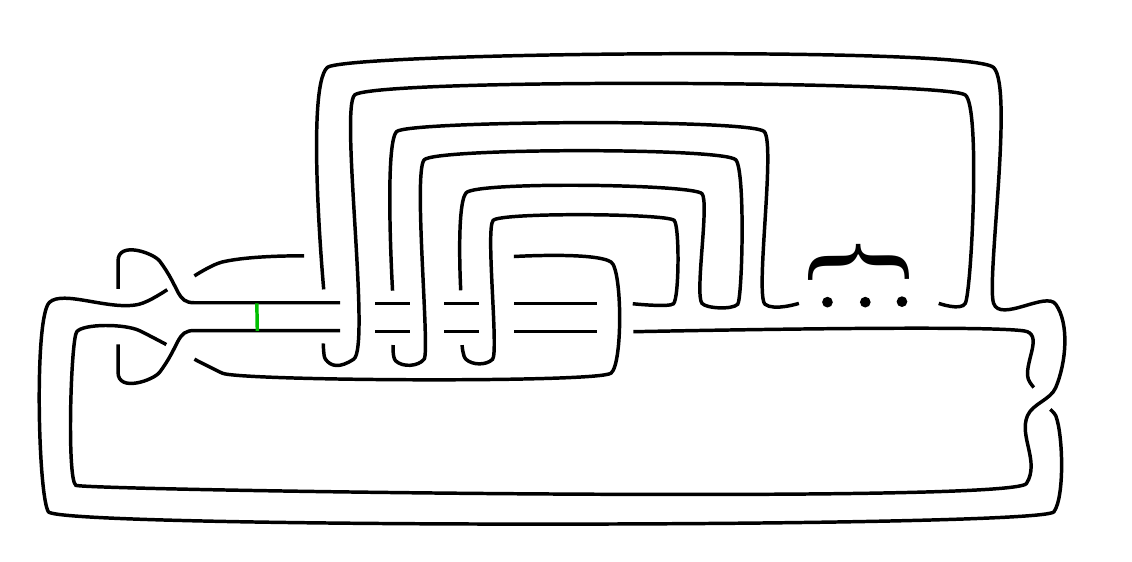}
\put(81,34){\colorbox{white}{\parbox{1.2cm}{%
    \tiny$n-times$}}}
\end{overpic}
\caption{An example ribbon knot with $n+2$ singularities for which a single
cut suffices.}
\label{single cut}

\end{figure}

This knot has $n+2$ ribbon singularities for any $n\in\mathbb{N}$,
and yet only one cut (shown in green) will reduce the picture to two
disjoint disks. 
\end{example}

Now we will introduce handles and obtain a Kirby picture in which
our knot $K$ takes a particularly simple form. We assume that the
reader is familiar with basic handlebody theory; an excellent reference
for this material is~\cite{key-4}. For every cut $c_{j}$, we will
attach an arc $h_{j}$ seen in Figure~\ref{2-handle for cut}. We
will think of $h_{j}$ as a thin ribbon, which would recover $K$
if glued along. For this reason we will give the arc, $h_{j}$, a
framing, by which we mean a parallel arc, and keep track of this framing
through any isotopies of $K$. By a \textit{1-sub-handlebody}, we
will mean the sub-handlebody consisting of the 0-handles and the 1-handles. 

\begin{figure}

\begin{overpic}[unit=1mm,scale=1]{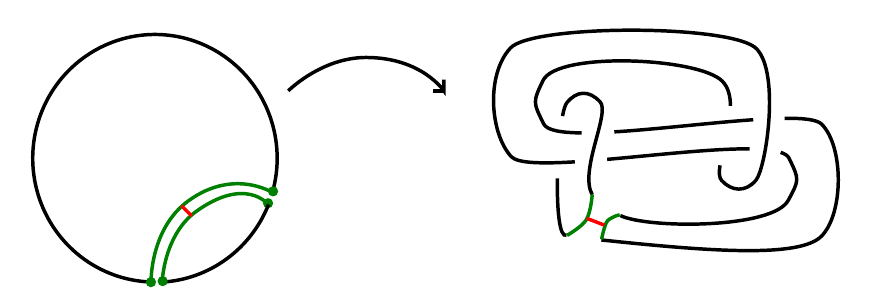}
\put(35,27){\colorbox{white}{\parbox{0.1cm}{%
    $\varphi$}}}
\put(54,3){\colorbox{white}{\parbox{0.1cm}{%
    \color{mygreen}\tiny$c_{j}$}}}
\put(61,11){\colorbox{white}{\parbox{0.1cm}{%
    \color{myred}\tiny$h_{j}$}}}
\end{overpic}
\caption{A 2-handle $h_{j}$ associated to a cut $c_{j}$}
\label{2-handle for cut}

\end{figure}

\subsubsection*{Proof of Theorem~\ref{unknot theorem } }

Using Theorem~\ref{thm:cutting to get disks}, we can make $k<n$
cuts to the ribbon disk to obtain the unlink. So we have a diagram
in which there are $k$ disjoint disks, and $k-1$ framed arcs $h_{j}$.
We know that by taking a band sum along these arcs (paying attention
to framings) we can recover our diagram for $K$. Let $K_{cut}$ be
the union of the boundaries of these disks. Now in a small neighborhood
of the end points of each $h_{j}$ we insert the attaching spheres
of a $1-$handle, letting $h_{j}$ be the attaching circle of a $2-$handle
as seen in Figure~\ref{1-2 handles}. 

\begin{figure}

\begin{overpic}[unit=1mm,scale=1]{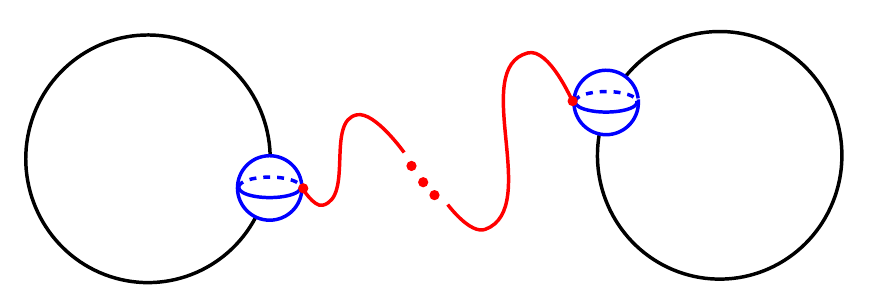}
\put(14,23){\colorbox{white}{\parbox{0.1cm}{%
    \footnotesize$K_{cut}$}}}
\put(70,23){\colorbox{white}{\parbox{0.1cm}{%
    \footnotesize$K_{cut}$}}}
\put(45,16){\colorbox{white}{\parbox{0.1cm}{%
    \color{myred}\footnotesize$h_{j}$}}}
\end{overpic}
\caption{A 1-2 handle canceling pair}
\label{1-2 handles}

\end{figure}

This pair cancels by construction, and also has the effect of doing
the band sum that recovers $K$ for the cut $c_{j}$ as seen in the
movie in Figure~\ref{movie}. Notice that we make two handle slides
that free $K_{cut}$ from the $1-$handle, and then cancel the pair.
Also notice that this has exactly the same effect that a band sum
of $K_{cut}$ along $h_{j}$ would have had. 

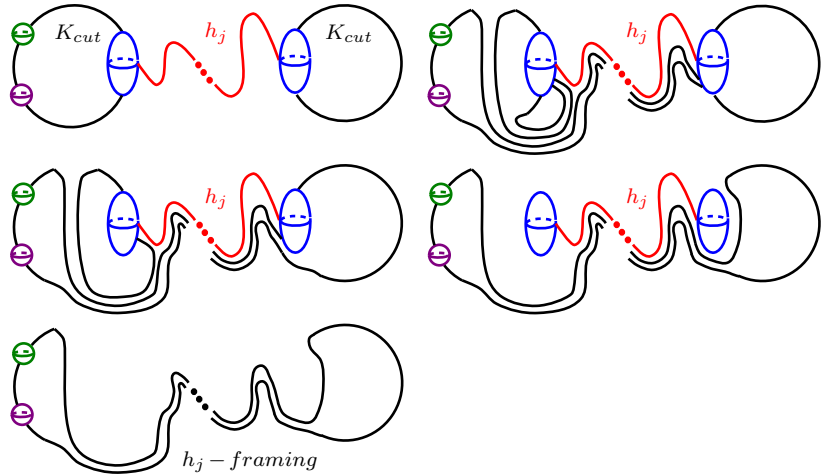
\begin{figure}

\begin{overpic}[unit=1mm,scale=1]{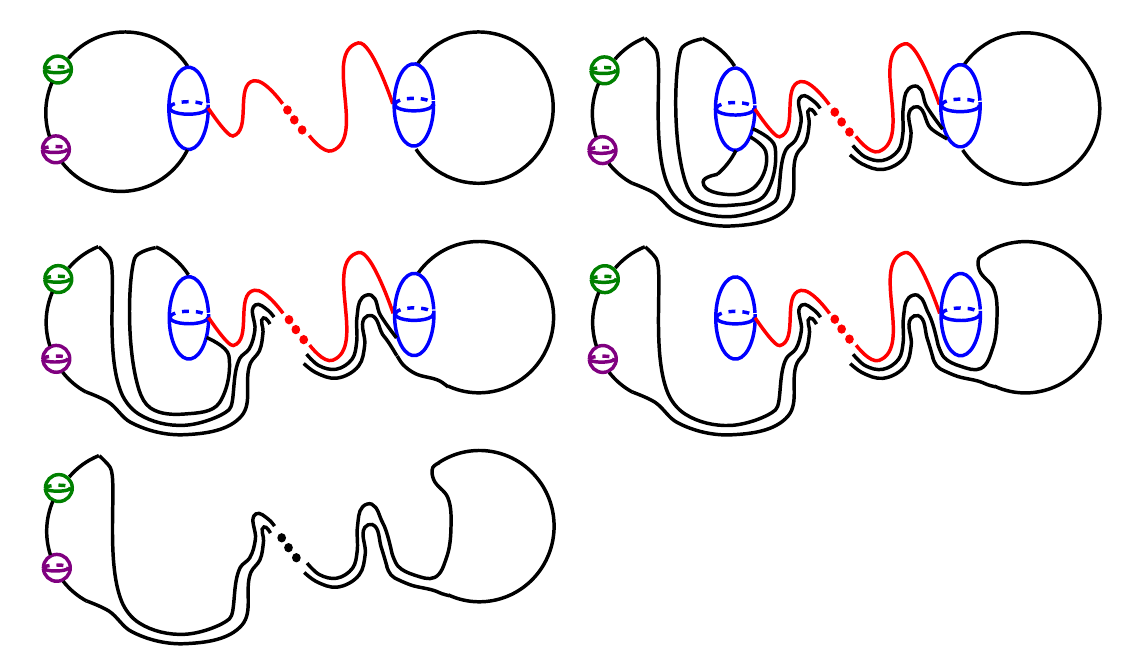}
\put(9,60){\colorbox{white}{\parbox{0.1cm}{%
    \tiny$K_{cut}$}}}
\put(45,60){\colorbox{white}{\parbox{0.1cm}{%
    \tiny$K_{cut}$}}}
\put(29,60){\colorbox{white}{\parbox{0.1cm}{%
    \color{myred}\tiny$h_{j}$}}}
\put(85,60){\colorbox{white}{\parbox{0.1cm}{%
    \color{myred}\tiny$h_{j}$}}}
\put(29,38){\colorbox{white}{\parbox{0.1cm}{%
    \color{myred}\tiny$h_{j}$}}}
\put(85,38){\colorbox{white}{\parbox{0.1cm}{%
    \color{myred}\tiny$h_{j}$}}}
\put(26,3){\colorbox{white}{\parbox{1.8cm}{%
    \tiny$h_{j}-framing$}}}
\end{overpic}
\caption{An example handle cancellation to recover $K$.}
\label{movie}

\end{figure}

There is no obstruction to this handle slide and cancellation caused
by the possible presence of other handle pairs, since the double band
sum shown on the left can be carried out in a small neighborhood of
the attaching sphere on the left. So after $n-1$ or less iterations,
we have recovered our diagram for $K$. It is worth noting that framings
on 2-handles denote an even number of half twists, therefore the framings
on the $h_{j}$ must be even. If our diagram for $K$ requires an
odd number of half twists then we can accommodate this by inserting
any number of half twists in one of the disks spanning $K_{cut}$
shown in figure Figure~\ref{2-handle adjustment} for the case of
a single half twist.

\begin{figure}

\begin{overpic}[unit=1mm,scale=1]{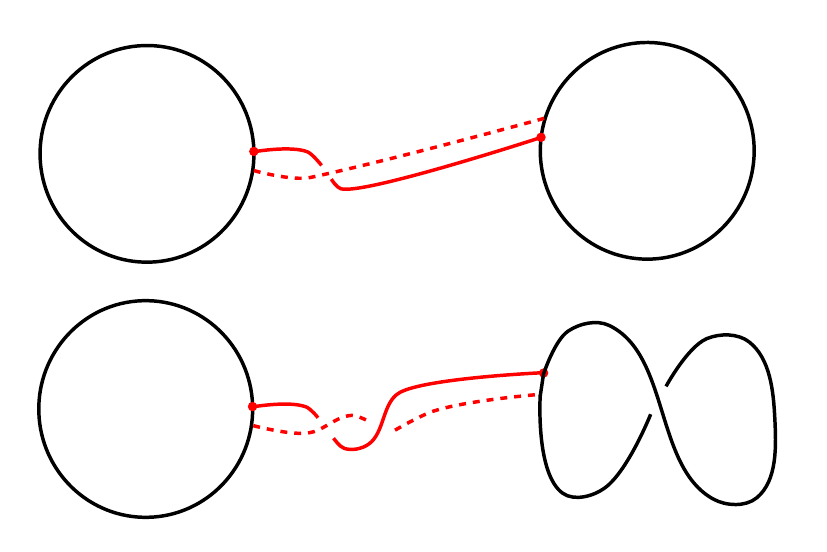}
\put(10,46){\colorbox{white}{\parbox{0.1cm}{%
    \footnotesize$K_{cut}$}}}
\put(65,46){\colorbox{white}{\parbox{0.1cm}{%
    \footnotesize$K_{cut}$}}}
\put(9,20){\colorbox{white}{\parbox{0.1cm}{%
    \footnotesize$K_{cut}$}}}
\put(69,24){\colorbox{white}{\parbox{0.1cm}{%
    \footnotesize$K_{cut}$}}}
\put(29,43){\colorbox{white}{\parbox{0.1cm}{%
    \color{myred}\footnotesize$h_{j}$}}}
\put(29,18){\colorbox{white}{\parbox{0.1cm}{%
    \color{myred}\footnotesize$h_{j}$}}}
\end{overpic}
\caption{Framing adjustment}
\label{2-handle adjustment}

\end{figure}
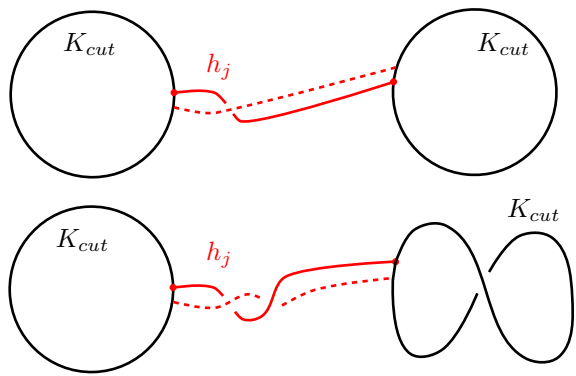

We would like to think of our diagram in which there are $k$ disjoint
disks connected by $k-1$ arcs, $h_{j}$, abstractly as a graph in
order to show that $K_{cut}$ can be pulled free of the 1-handles.
To do this, we first work in the boundary of the 1-sub-handlebody.
We think of each of our disjoint disks as a vertex, and put an edge
between vertices if the corresponding disks are joined by a 1-handle.
Notice $G$ embeds in $D_{K}$ as the ``dual'' graph to $D_{K}$
cut along $\varphi^{-1}\left(c_{j}\right)$, that is, there is a vertex
in the center of each component of $D_{K}-\cup_{j=1}^{k-1}\varphi^{-1}\left(c_{j}\right)$
and an edge for each $\varphi^{-1}\left(c_{j}\right)$. Then $G$
is homotopy equivalent to $D_{K}$, and so we see that $\chi\left(G\right)=\chi\left(D_{K}\right)=1$.
It is well known that the Euler characteristic of a connected graph
is one if and only if that graph is a tree, so $G$ is a tree. Each
uni-valent vertex of $G$ is now associated to a portion of our picture
consisting of two disks connected by a $1-$handle, where, one disk
might have many $1-$handle attaching spheres, but the other must
have exactly one $1-$handle attaching sphere as shown in Figure~\ref{uni-valent vertex}.
In the 1-sub-handlebody it is clear that $K_{cut}$ may be isotoped
off this 1-handle. Notice that the effect of this isotopy on $G$
is to remove the corresponding edge and uni-valent vertex from the
graph. Since $G$ is a tree, we can iterate this procedure revealing
that $K_{cut}$ can be pulled completely free of the 1-handles. This
may be seen in Figure~\ref{2-handlebody} by simply ignoring the
attaching circles of the 2-handles, $h_{j}$. 

\begin{figure}

\begin{overpic}[unit=1mm,scale=1]{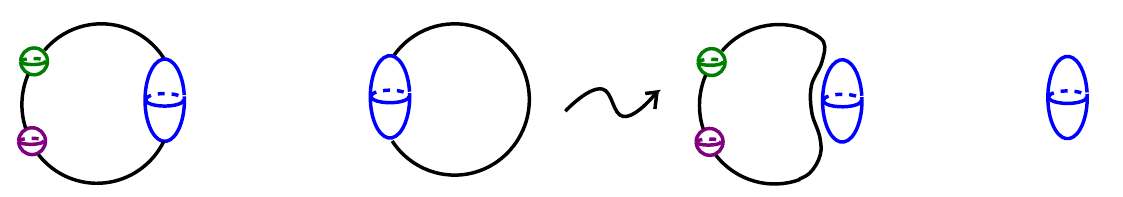}
\put(6,13){\colorbox{white}{\parbox{0.1cm}{%
    \tiny$K_{cut}$}}}
\put(43,13){\colorbox{white}{\parbox{0.1cm}{%
    \tiny$K_{cut}$}}}
\put(75,13){\colorbox{white}{\parbox{0.1cm}{%
    \tiny$K_{cut}$}}}
\end{overpic}
\caption{Handle picture corresponding to a uni-valent vertex of $G$.}
\label{uni-valent vertex}
\end{figure}

The above iteration gives an isotopy of $K_{cut}$ which extends to
an ambient isotopy of the boundary of the 1-sub-handlebody. This,
in turn, induces an isotopy on the attaching circles of the 2-handles,
$h_{j}$, resulting in a 2-handlebody as claimed in Theorem~\ref{unknot theorem }.
See Figure~\ref{2-handlebody}. By construction, handle slides and
cancellations gives us a knot isotopic to $K\subset S^{3}$. \textifsymbol[ifgeo]{48}

\begin{figure}

\begin{overpic}[unit=1mm,scale=1]{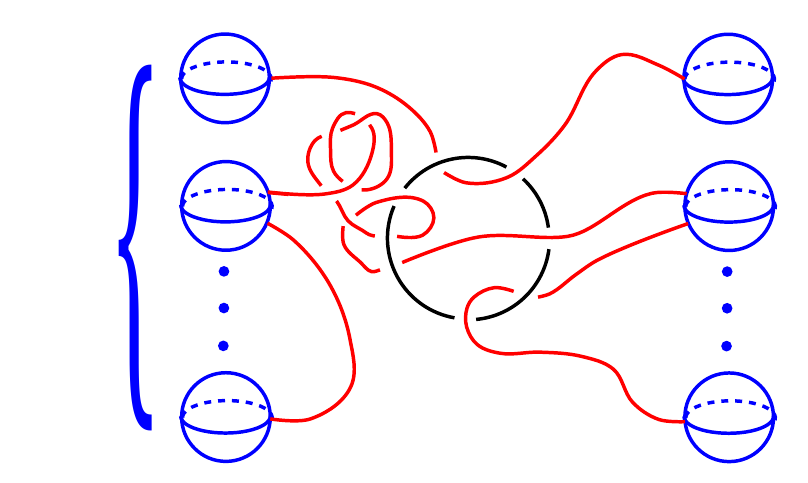}
\put(39,14){\colorbox{white}{\parbox{0.1cm}{%
    $K$}}}
\put(-7,23){\colorbox{white}{\parbox{1.6cm}{%
    \color{myblue}\tiny$n-1\;or\;less$}}}
\end{overpic}
\caption{A 2-handlebody picture where $K$ appears as the unknot in the boundary
of the 1-sub-handlebody.}
\label{2-handlebody}
\end{figure}

So we have shown that any ribbon knot with $n$ ribbon singularities
may be constructed by starting with the unknot in $\underset{k}{\#}S^{1}\times S^{2}$,
where $k\leq n-1$, and attaching 2-handles to cancel each of the
1-handles in an appropriate manner. 
\begin{example}
It is an exercise in Kirby calculus to show that images in Figure~\ref{example repeated}
are two pictures of the same ribbon knot in $S^{3}$. 
\end{example}

\begin{cor}
In Figure~\ref{2-handlebody}, if we replace the unknot in the 1-sub-handlebody
with a dotted circle, then we obtain a picture of the 4-manifold which
is the complement of the slicing disk in $D^{4}$, shown in Figure~\ref{slice disk complement}.\label{slice disk complement cor}
\end{cor}

\begin{figure}

\includegraphics{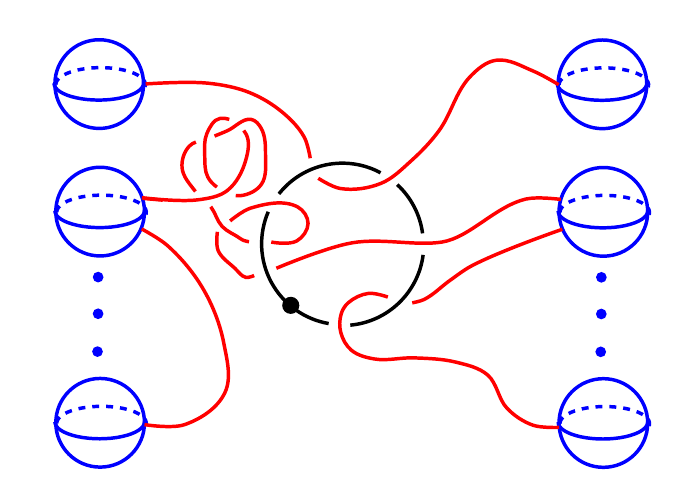}

\caption{A 2-handlebody picture of the complement of the slice disk for $K$. }
\label{slice disk complement}

\end{figure}

\begin{proof}
The slicing disk can be seen in the picture as the disk filling the
unknot that we have in the 1-sub-handlebody. This is since canceling
the 1-2 handle pairs not only recovers $K$, but also recovers the
ribbon disk $\widetilde{D}_{K}$. The definition of the dotted circle
notation is that we remove a small neighborhood of the dotted unknot
along with a small neighborhood of the disk after pushing it into
$D^{4}$. And so this is exactly the complement of the slicing disk,
$D^{4}-\widetilde{D}_{K}$.
\end{proof}
One nice fact is that, since disk slides, isotopies and handle cancellations
can be done locally, and since ribbon knots always bound an immersed
ribbon disk, this construction actually works in any 3-manifold. We
did not rely on any special properties of $S^{3}$ during the process.
One can create examples by combining a 2-handlebody picture for a
ribbon knot $K\subset S^{3}$ as in the above construction with a
Kirby picture of a 4-manifold $W$ whose boundary is the intended
3-manifold $M^{3}=\partial W$. When combining the two pictures, $K$
may be allowed to run across non-canceling 1-handles to form non-trivial
examples as shown below. In Figure~\ref{s1_cross_s2} we have a Kirby
picture of a 4-manifold whose boundary is $S^{1}\times S^{2}$. We
can see the ribbon disk for $K$ in the image on the left. The image
on the right shows the result using the technique developed above. 

\begin{figure}

\begin{overpic}[unit=1mm,scale=1]{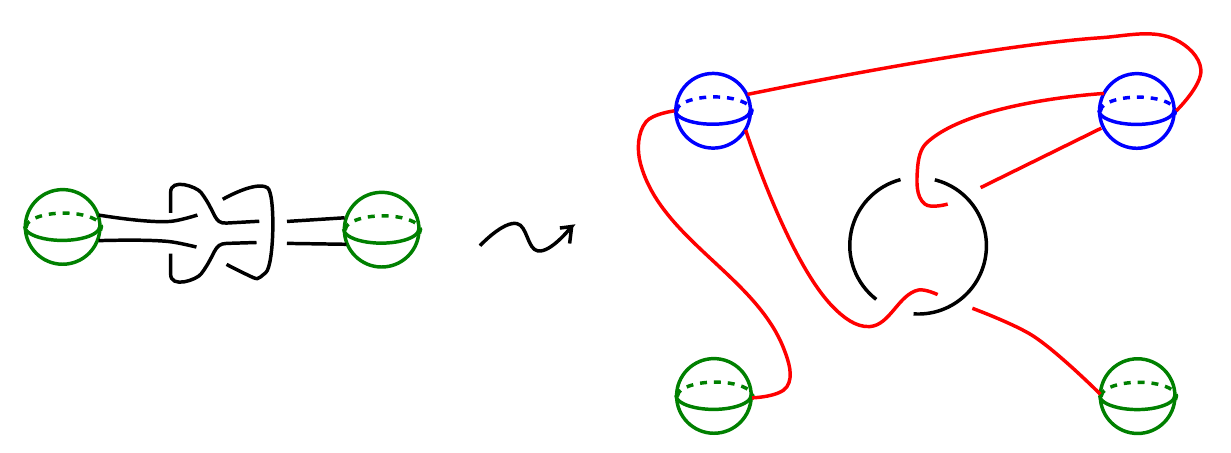}
\put(101,21){\colorbox{white}{\parbox{0.1cm}{%
    $K$}}}
\put(20,30){\colorbox{white}{\parbox{0.1cm}{%
    $K$}}}
\end{overpic}
\caption{An example ribbon knot in $S^{1}\times S^{2}$ and its decomposition.}
\label{s1_cross_s2}
\end{figure}


\begin{thebibliography}{Eliash}
\bibitem[EH1]{key-1}J. B. Etnyre and K. Honda, \textit{Cabling and
Transverse Simplicity}, Annals of Math 162(2005), 1305-1333.

\bibitem[EH2]{key-2}J. B. Etnyre and K. Honda, \textit{Knots and
Contact Geometry}, J. Symplectic Geom. 1 (2001), 63\textendash 120;
math/0006112 .

\bibitem[ELT]{key-3}J. B. Etnyre, D. J. LaFountain, B. Tosun, \textit{Legendrian
and transverse cables of positive torus knots}, Geometry \& Topology
16 (2012) 1639\textendash 1689.

\bibitem[GS]{key-4}R.E. Gompf and A.I. Stipsicz, \textit{4-Manifolds
and Kirby Calculus}, American Mathematical Society. 1999.

\bibitem[H]{key-5}K. Honda, \textit{On the Classification of Tight
Contact Structures I}, Geom. Topol. 4 (2000) 309-368.

\bibitem[Y]{key-6}K.Yasui, \textit{Maximal Thurston-Bennequin Number
and Reducible Legendrian Surgery}, Compositio Mathematica, Volume
152, Issue 9, September 2016 , pp. 1899-1914.

\bibitem[K]{key-7}Y. Kanda, \textit{The classification of tight contact
structures on the 3-torus}, Comm. Anal. Geom., \textbf{5} (1997),
413-438.

\bibitem[G]{key-8}E. Giroux , \textit{Convexit\'{e} en topologie de contact}
, Comm. Math. Helv. 66 (1991) 637\textendash 677.

\bibitem[EV]{key-9}J. B. Etnyre and V. V\'{e}rtesi, \textit{Legendrian
Satellites}, International Mathematics Research Notices, rnx106.

\bibitem[LS]{key-10}T. Lidman and S. Sivek, \textit{Contact structures
and reducible surgeries}, Compositio Mathematica, Volume 152, Issue
1, January 2016, pp. 152-186.

\bibitem[CET]{key-11} J. Conway, J. B. Etnyre, B. Tosun, \textit{Symplectic
fillings, contact surgeries, and Lagrangian disks}, arXiv:1712.07287
{[}math.GT{]}.

\bibitem[Hed]{key-12}M. Hedden, \textit{NOTIONS OF POSITIVITY AND
THE OZSV\'{A}TH-SZAB\'{O} CONCORDANCE INVARIANT}, Journal of Knot Theory Ramifications,
19, 617 (2010). 

\bibitem[Eliash]{key-13}Y Eliashberg , \textit{Classification of
overtwisted contact structures on 3\textendash manifolds} , Invent.
Math. 98 (1989) 623\textendash 637

\bibitem[Min]{key-14}H. Min, \textit{Knots and uniform thickness
property} (work in progress).
\end{thebibliography}
\end{document}